\documentclass[final]{siamart171218}
\usepackage[english]{babel}
\usepackage{times}
\usepackage{comment}
\usepackage{amsopn,amsfonts,graphicx,cite,cleveref,algorithm,algorithmic}

\newcommand{\tmop}[1]{\ensuremath{\operatorname{#1}}}
\def \R {\mathbb{R}}
\def \N {\mathbb{N}}
\newsiamthm{cor}{Corollary}%[section]
\newsiamthm{prop}{Proposition}%[section]
\newsiamthm{lem}{Lemma}
\newsiamthm{defi}{Definition}

\theoremstyle{plain}

\newtheorem{assump}{Assumption}

\DeclareMathOperator*{\minimize}{\tmop{minimize}}

\DeclareMathOperator*{\argmin}{\tmop{arg~min}}

\def \st {\operatorname*{subject\ to\ }}
\newcommand{\rank}{\operatorname{rank}}
\def \lg        {\langle}
\def \rg        {\rangle}
\def \eye       {\mathbf{I}}
\def \zero      {\mathbf{0}}
\def \vzero     {\vct{0}}

\def \Sp        {\tmop{Sp}}

\newcommand{\vct}[1]{\boldsymbol{#1}}
\newcommand{\mtx}[1]{\mathbf{#1}}
\newcommand{\set}[1]{\mathcal{#1}}

\newcommand{\mPhi}{\mathbf{\Phi}}

\def \a {\vct{a}}

\def \y {\vct{t}}

\def \x {\vct{x}}
\def \y {\vct{y}}

\newcommand{\vu}{\vct{u}}

\newcommand{\vx}{\vct{x}}
\newcommand{\vy}{\vct{y}}

\newcommand{\mA}{\mtx{A}}

\newcommand{\mF}{\mtx{F}}

\newcommand{\mH}{\mtx{H}}
\newcommand{\mJ}{\mtx{J}}

\newcommand{\mU}{\mtx{U}}
\newcommand{\mV}{\mtx{V}}

\newcommand{\mX}{\mtx{X}}

\def \calA {\set{A}}
\def \calB {\set{B}}
\def \calC {\set{C}}

\def \calN {\set{N}}

\usepackage{colortbl, booktabs,pifont}
\newcommand{\blk}[2]{  
    \begin{minipage}{#1\textwidth} 
        \centering \vspace*{0.08cm}
        #2 \vspace*{0.05cm}     
\end{minipage}}
\newcommand{\cmark}{\ding{51}}%
\newcommand{\xmark}{\ding{55}}%

\title{Provable Bregman-divergence based Methods for Nonconvex and Non-Lipschitz Problems\thanks{Submitted to the editors April 17, 2019. 
The first and second authors contributed equally to this paper. 
\funding{This work was supported by Award N660011824020 from the DARPA Lagrange Program.}}}

\author{Qiuwei Li\thanks{Department of Electrical Engineering, Colorado School of Mines.
(\email{qiuli@mines.edu}, \url{http://inside.mines.edu/\string~qiuli/};\email{gtang@mines.edu}, \url{http://inside.mines.edu/\string~gtang/}; \email{mwakin@mines.edu}, \url{http://inside.mines.edu/\string~mwakin/}).}
\and Zhihui Zhu\thanks{Mathematical Institute for Data Science, Johns Hopkins University. 
(\email{zzhu29@jhu.edu}, \url{http://cis.jhu.edu/\string~zhihui/}).}
\and Gongguo Tang\footnotemark[2]
\and Michael B. Wakin\footnotemark[2]}

\begin{document}

\allowdisplaybreaks
\maketitle

\begin{abstract}%
The (global) Lipschitz smoothness condition is crucial in establishing the convergence theory for most optimization methods. Unfortunately,  most machine learning and signal processing problems are not Lipschitz smooth. This motivates us to generalize the concept of Lipschitz smoothness condition to the relative smoothness condition, which is satisfied by any finite-order polynomial objective function. Further, this work develops new Bregman-divergence based algorithms that are guaranteed to converge to a second-order stationary point for any relatively smooth problem. In addition, the proposed optimization methods cover both the proximal alternating minimization and the proximal alternating linearized minimization when we specialize the Bregman divergence to the Euclidian distance. Therefore, this work not only develops guaranteed optimization methods for non-Lipschitz smooth problems but also solves an open problem of showing the second-order convergence guarantees for these alternating minimization methods.
\end{abstract}

% REQUIRED

% REQUIRED

\section{Introduction}
Consider minimizing a twice continuously differentiable function
\begin{align}
\minimize_{\vx\in\R^n}  f(\vx)
\label{problem}
\end{align}
which can be solved by numerous off-the-shelf algorithms, such as first-order algorithms like  gradient descent  \cite{lee2017first}, perturbed gradient descent \cite{jin2017escape}, nonlinear conjugate gradient method \cite{dai1999nonlinear},  proximal point minimization algorithm \cite{nocedal2006numerical}, and second-order methods like the Newton-CG algorithms \cite{Royer2018,royer2018newton}.

\subsection{The Lipschitz Smoothness Condition}
However, all these optimization algorithms require the objective function $f(\x)$ to be Lipschitz smooth.  When $f$ is twice continuously differentiable, denoted by $f\in\calC^2$,   we say $f$ is {\em $L_f$-Lipschitz smooth} if there exists a constant $L_f$ so that 
\begin{align}
L_{f}\eye\pm \nabla^{2}f(\x) \succeq 0,\quad \forall~ \x\in\R^n
\label{lipschitz}
\end{align}
The Lipschitz smoothness condition plays a fundamental rule in establishing the convergence theory for most optimization methods, including both the first-order and the second-order convergence guarantee. Here, we refer the first-order (or second-order) convergence to the convergence to a first-order (or second-order) stationary point (cf. \Cref{def:critical:point}). For an illustration, let us consider the gradient descent (GD) algorithm, which iteratively updates the sequence along the negative gradient by an amount of a constant step. Mathematically, GD generates a sequence $\{\x_k\}$ by iteratively computing
\[
\x_{k+1}=\x_{k}-\eta \nabla f(\x_{k}), \quad k=1,2,\ldots
\]
The main purpose of using the Lipschitz-smoothness condition is to guarantee a sufficient decrease of the objective function in each step. 
\begin{lemma}
Assume $f$ is $L_f$-Lipschitz smooth. Then GD with step size $\eta$ satisfies 
\begin{align}
f(\vx_{k}) - f(\vx_{k+1}) \ge\left(\frac{1}{\eta} - \frac{L_f}{2}\right)\|\vx_{k} - \vx_{k+1}\|_2^{2}
\label{eqn:descent:lem:gd}
\end{align} 
\label{lem:GD}   
\end{lemma}
\begin{proof}
This  follows by  plugging  $\y=\x_{k},~\x=\x_{k+1}$ to 
a consequence of the Lipschitz smoothness condition: $\left| f(\x)-f(\y)- \lg\nabla f(\y),\x-\y\rg \right|\le \frac{L_{f}}{2} \|\x-\y\|_{2}^{2}$ for all $\x,\y$.
\end{proof}
   
%Therefore, to make GD decrease the function values, we should choose the step size $\eta<2/L_f$. However, in many cases, the target problem is not Lipschitz smooth. Even if we assume  GD converges to a first-order critical point,  we may meet situations where the iterative algorithms bump into saddle points. The recent  work \cite{lee2017first} shows that most first-order methods  almost surely avoid strict saddles (cf. \Cref{def:critical:point}) from random initialization. However, one pivotal assumption in their argument is still the Lipschitz smoothness condition of the objective function. It is not well understood how the first-order methods can avoid saddles when the objective function is not Lipchitz smooth.

Therefore, as long as the step size is small enough $\eta<2/L_f$, GD makes a decrease of the function value in every step and converges to a first-order stationary point.  Besides, the recent work  \cite{lee2017first} establishes that most first-order algorithms (e.g., GD) can avoid strict saddle points (those first-order stationary points that are not second-order stationary points) for any Lipschitz-smooth problem. Therefore, in addition to its importance in building up the first-order convergence guarantees, the Lipschitz smoothness condition is also pivotal in demonstrating the second-order convergence theories.

Unfortunately,   the objective functions in many machine learning problems---such as low-rank matrix recovery, tensor factorization problem,  neural networks training---is not Lipschitz smooth. This is because, by its definition, the Lipschitz smoothness condition requires the spectral norm of the Hessian matrix is in a constant order like $O(1)$. However, most machine learning problems possess a high-degree (higher than quadratic) polynomial objective function, and as a result, the spectral norm of the  Hessian matrices will be in a polynomial order (e.g., $O(x^2)$) and thus fails to be in a constant order.

\subsection{The Relative Smoothness Condition} 
It is not quite understood how to deal with these non-Lipschitz smooth but popular machine learning problems. This prompts us to generalize the idea of the Lipschitz smoothness condition to the relative smoothness condition by considering a generalized smoothness using the Bregman divergence (cf. \cite{bolte2018first}). 

\begin{defi}[Bregman Divergence]
The Bregman divergence, in term of a strongly convex function $h$, is defined as
\begin{align}
D_h(\vx,\vy)=h(\vx)- h(\vy) - \langle\nabla h(\vy),\vx-\vy\rangle
\end{align}
\label{def:bregman:distance}
\end{defi}

\vspace{-1em}

When $h(\vx) = \|\vx\|_{2}^2/2$, the Bregman divergence reduces to $D_h(\vx,\vy) = \|\vx - \vy\|_{2}^2/2$, the Euclidean distance. 

\begin{defi}[Relative Smoothness]
\label{defi:adative:lipschitz}
A function $f$ is   $L_{f}$-relative smooth  if there exists a   strongly convex function $h$ such that
\begin{align}
\label{eqn:general:lipschitz}
L_{f} \nabla^{2}h(\x)\pm \nabla^{2}f(\x)\succeq 0,\quad\forall~\x\in\R^{n}
\end{align}
\end{defi}
When $h(\vx) = \|\vx\|_{2}^2/2$, the relative smoothness condition reduces to the Lipschitz smoothness condition.

The relative smoothness condition plays a similarly significant role in establishing the convergence guarantees for relatively smooth problems as the Lipschitz smoothness condition does for Lipschitz smooth problems. 
To see this, a direct consequence of the relative smoothness condition is the following generalized descent lemma: 
\begin{align}
|f(\x)-f(\y)-\lg \nabla f(\y),\x-\y\rg | \le L_{f} D_{h} (\x,\y),\quad\forall~\x,\y\in\R^{n}
\label{eqn:general:descent:lemma}
\end{align}
which is essential in developing the convergence theory of the proposed Bregman-divergence based methods  (cf. \Cref{sec:algs}) to solve those non-Lipschitz smooth but relative smooth problems.

\begin{lem}\label{lem:general:decrease}
Suppose $f\in\mathcal{C}^{2}$ is lower-bounded and $L_f$-relatively smooth w.r.t. some  $\sigma$-strongly convex and super-coercive\footnote{We say $h$ is super-coercive if $\lim_{\|\vx\|\rightarrow\infty}{h(\vx)}/{\|\vx\|_{2}} = \infty$, and coercive if $\lim_{\|\vx\|\rightarrow\infty}{h(\vx)}= \infty$.} function $h\in \mathcal{C}^{2}$.    Then each iteration of Bregman gradient descent (B-GD; \Cref{alg:bgd}) or   Bregman proximal point minimization (B-PPM; \Cref{alg:bpd}) is  well-defined, and  respectively satisfies
\begin{align}
\tmop{(B-GD)}\quad& f(\x_{k-1})-f(\x_{k})\ge \left(\frac{1}{\eta}-L_{f}\right)\frac{\sigma}{2} \|\x_{k}-\x_{k-1}\|_{2}^{2}
\label{eqn:general:decrease1},
\\
\tmop{(B-PPM)}\quad&f(\x_{k-1})-f(\x_{k})\ge \frac{\sigma}{2\eta}\|\x_{k}-\x_{k-1}\|_{2}^{2}
\label{eqn:general:decrease2}
\end{align}
\end{lem}
The proof is in \Cref{pf:lem:general:decrease}.
Note that since the Bregman divergence can be specialized as the Euclidean distance (when $h(\x)=\|\x\|_2^2$),  the Bregman-divergence based algorithms: B-GD and B-PPM are extensions of the standard Euclidean-distance based algorithms: GD and proximal point minimization algorithm (PPM) \cite{nocedal2006numerical,lee2017first}.

\subsection{Extension to Alternating Minimizations}
A natural way to solve factored matrix optimization problems like $\min_{\mU, \mV} f(\mU\mV^T)$  is via alternating minimization with matrices $\mU$ and $\mV$. However, the state-of-the-art alternating minimization methods (e.g.,   the proximal alternating linearized minimization (PALM)  \cite{bolte2014proximal} and the proximal alternating minimization (PAM) \cite{attouch2010proximal,xu2013block}) can only ensure the first-order convergence and require the objective function $f$ to be Lipschitz smooth. This leads to two main drawbacks of applying these algorithms to the factored matrix optimization problem: 1) even it is recognized that (cf. \cite{ge2017no,li2018non}) many factored matrix optimization problems  has the nice property that any second-order stationary point is globally optimal, the current first-order convergence of these state-of-the-art alternating minimization algorithms cannot help to achieve the global optimal solution; 2) most factored matrix optimization problems are known to be non-Lipschitz smooth, therefore, even the current first-order convergence guarantees cannot apply to these problems. 

This motivates us to generalize the previous Bregman optimization algorithms (e.g., B-GD and B-PPM) that can naturally solve the following  alternating minimization problem
\begin{align}
\minimize_{\x\in\R^n,\y\in\R^m}f(\x,\y)
\label{eqn:problem:alt}
\end{align}
We will call  $f(\x,\y)$ as a bi-variable function for convenience.
Fortunately, similar sufficient decrease property can be established for the Bregman alternating minimization methods  (cf. \Cref{sec:algs}) to solve the non-Lipschitz smooth but the relatively smooth problem. Towards that end, we need similar concepts like bi-Bregman divergence and relative bi-smoothness condition.
\begin{defi}[Bi-Bregman Divergence]
The bi-Bregman divergences, in term of a strongly bi-convex\footnote{$h(\x,\y)$ is strongly bi-convex if it is a strongly convex function in either variable when fixing the other.} function $h(\x,\y)$,   are defined as
\begin{align}
D_h^1(\vx_1,\vx_2;\y)=&h(\vx_1,\y)- h(\vx_2,\y) - \langle\nabla_{\x} h(\vx_2,\y),\vx_1-\vx_2\rangle,
\\
D_h^2(\vy_1,\vy_2;\x)=&h(\vx,\y_1)- h(\vx,\y_2) - \langle\nabla_{\y} h(\vx,\y_2),\vy_1-\vy_2\rangle
\end{align}
\label{def:bi:bregman:distance}
\end{defi}

\vspace{-1em}

When $h(\x,\y)=(\|\x\|_2^2+\|\y\|_2^2)/2$, the above two bi-Bregman divergences will reduce to the Euclidian distances $\|\x_1-\x_2\|_2^2/2$ and  $\|\y_1-\y_2\|_2^2/2$, respectively. 

\begin{defi}[Relative Bi-Smoothness]
\label{defi:adative:lipschitz:coordinate}
$f$ is $(L_1,L_{2})$-relative bi-smooth if there exists a strongly bi-convex function  $h$ such that
 \begin{align}
\label{eqn:general:lipschitz:coordinate}
L_{1} \nabla^{2}_{\x\x}h(\x,\y) \pm \nabla_{\x\x}^{2}f(\x,\y)\succeq 0,~ L_{2} \nabla^{2}_{\y\y}h(\x,\y) \pm \nabla_{\y\y}^{2}f(\x,\y)\succeq 0,~\forall~ \x,\y
\end{align}
\end{defi}
We now provide the sufficient decrease property for Bregman alternating minimizations.
\begin{lem}\label{lem:general:decrease:coordinate}
Suppose $f(\x,\y)\in\mathcal{C}^{2}$ is  lower-bounded and $L_f$-relatively smooth w.r.t. some  $\sigma$-strongly bi-convex and bi-super-coercive\footnote{
$h(\x,\y)$ is  bi-super-coercive if $h(\x,\y)$ is super-coercive in either variable when fixing the other.} function $h(\x,\y)\in \mathcal{C}^{2}$.
Then each iteration of Bregman proximal alternating linearized minimization (B-PALM; \Cref{alg:balsg}) or  Bregman proximal alternating minimization (B-PAM; \Cref{alg:bals}) is  well-defined and  respectively satisfies
\begin{align} 
\tmop{(B-PALM)}\quad
&f(\x_{k\!-\!1},\!\y_{k\!-\!1})\!-\!f(\x_{k},\!\y_{k})\ge \left(\frac{1}{\eta}\!-\!L_{f}\right)\frac{\sigma}{2} \|(\x_{k},\!\y_{k})\!-\!(\x_{k\!-\!1},\!\y_{k\!-\!1})\|_{2}^{2}
\label{eqn:general:decrease2:coordinate:1},
\\
\tmop{(B-PAM)}\quad
&f(\x_{k\!-\!1},\!\y_{k\!-\!1})\!-\!f(\x_{k},\!\y_{k})\ge \frac{\sigma}{2\eta}\|(\x_{k},\y_{k})\!-\!(\x_{k\!-\!1},\y_{k\!-\!1})\|_{2}^{2}\label{eqn:general:decrease2:coordinate}
\end{align}
\end{lem}
The proof of \Cref{lem:general:decrease:coordinate} follows  by  applying \Cref{lem:general:decrease} for two times. 
Similarly, since the bi-Bregman divergences can be specialized as the Euclidean distances, B-GD and B-PPM are generalizations of the standard Euclidean-distance based proximal alternating minimization algorithms:  PALM  \cite{bolte2014proximal} and PAM \cite{attouch2010proximal,xu2013block}, respectively. 

\subsection{Main Results}
\label{sec:algs}

This work provides the second-order convergence guarantees for these four Bregman-divergence based algorithms to deal with non-Lipschitz smooth but relative smooth problems. 

\begin{algorithm}[H]
\caption{Bregman Gradient Descent (B-GD)}
\label{alg:bgd}
\begin{algorithmic}
\STATE \textbf{Input:}  Some $h(\x)\in\calC^2$ so that $f(\x)$ is $L_f$-relatively smooth w.r.t. $h(\x)$

\STATE \textbf{Initialization:}  $\vx_0$

\STATE \textbf{Recursion:} Set $\eta\in(0,{1}/{L_{f} })$ and
 iteratively generate  $\{\x_{k}\}_{k\in\N}$ via
\begin{align}
\label{eqn:bgd1}
\x_{k}=\argmin_{\vx} \langle\nabla f(\x_{k-1}),\x-\x_{k-1}\rangle+\frac{1}{\eta}D_h(\x,\x_{k-1})
\end{align}
\vspace{-1em}
\end{algorithmic}
\end{algorithm}

\vspace{-1em}

\begin{algorithm}[H]
\caption{Bregman Proximal Point Minimization (B-PPM)}
\label{alg:bpd}
\begin{algorithmic}
    \STATE \textbf{Input:}   Some $h(\x)\in\calC^2$ so that $f(\x,\y)$ is $L_f$-relatively smooth w.r.t. $h(\x,\y)$

    \STATE \textbf{Initialization:}  $\vx_0$

    \STATE \textbf{Recursion:} Set $\eta\in(0,\frac{1}{L_{f} })$ and tteratively generate  $\{\vx_{k}\}_{k\in\N}$ via
\begin{align}
\x_{k}=\argmin_{\vx}  f(\x)+\frac{1}{\eta}D_h(\x,\x_{k-1})
\label{eqn:bpd1}
\end{align}
\vspace{-1em}
\end{algorithmic}
\end{algorithm}

\vspace{-1em}

\begin{algorithm}[H]
\caption{Bregman Proximal Alternating Linearized Minimization (B-PALM)}
\label{alg:balsg}
\begin{algorithmic}
\STATE \textbf{Input:} Some $h(\x,\y)\in\calC^2$ so that $f(\x)$ is $(L_1,L_2)$-relatively bi-smooth w.r.t. $h(\x,\y)$

\STATE \textbf{Initialization:}  $(\x_{0},\y_{0})$

\STATE \textbf{Recursion:} Set  $\eta\in(0,1/\max\{L_{1},L_2\})$ and iteratively generate  $\{\x_{k},\y_{k}\}_{k\in\N}$ via
\begin{equation}
\begin{aligned}
\x_{k}=& \argmin_{\x} \lg \nabla_{\x} f(\x_{k-1},\y_{k-1}), \x-\x_{k-1}\rg+\frac{1}{\eta}D^1_{h}(\x,\x_{k-1};\y_{k-1}),\\
\y_{k}=& \argmin_{\y} \lg\nabla_{\y} f(\x_{k},\y_{k-1}),\y-\y_{k-1}\rg+\frac{1}{\eta}D^2_{h}(\y,\y_{k-1};\vx_{k})
\end{aligned}
\label{update:gd:coordinate}
\end{equation}
\vspace{-1em}
\end{algorithmic}
\end{algorithm}

\vspace{-1em}

\begin{algorithm}[H]
\caption{Bregman Proximal Alternating Minimization (B-PAM)}
\label{alg:bals}
\begin{algorithmic}
\STATE \textbf{Input:}  Some $h(\x,\y)\in\calC^2$ so that $f(\x,\y)$ is $(L_1,L_2)$-relatively bi-smooth w.r.t. $h(\x,\y)$

    \STATE \textbf{Initialization:}  $(\x_{0},\y_{0})$

     \STATE \textbf{Recursion:} Set  $\eta\in(0,1/\max\{L_{1},L_2\})$ and iteratively generate  $\{\vx_{k},\y_{k}\}_{k\in\N}$ via
\begin{equation}
\begin{aligned}
\x_{k}=& \argmin_{\x} f(\x,\y_{k-1})+\frac{1}{\eta}D^1_{h}(\x,\x_{k-1};\y_{k-1}),\\
\y_{k}=& \argmin_{\y}  f(\x_{k},\y)+\frac{1}{\eta}D^2_{h}(\y,\y_{k-1};\vx_{k})
\end{aligned}
\label{update:proximal:coordinate}
\end{equation}
\vspace{-1em}
\end{algorithmic}
\end{algorithm}

\vspace{-0.5em}

It is worth reminding  that these Bregman-divergence based algorithms are generalizations of those traditional Euclidian-distance based  algorithms to deal with non-Lipschitz smooth problems:
\begin{itemize}
\item[1)] B-GD (\Cref{alg:bgd}) generalizes the traditional Euclidian-distance based gradient descent algorithm (GD) \cite{lee2017first};
\item[2)] B-PPM (\Cref{alg:bpd}) generalizes the  traditional Euclidian-distance basedproximal point minimization algorithm (PPM) \cite{nocedal2006numerical,lee2017first};
\item[3)] B-PALM (\Cref{alg:balsg}) generalizes the traditional Euclidian-distance based proximal alternating linearized minimization algorithm (PALM) \cite{attouch2010proximal,bolte2014proximal}; 
\item[4)] B-PAM  (\Cref{alg:bals}) generalizes the traditional Euclidian-distance based proximal alternating minimization algorithm (PAM) \cite{xu2013block}.
\end{itemize}
Note that although B-GD has been studied in the previous work \cite{bolte2018first}, this work complements their first-order convergence theory by providing the second-order convergence theory for B-GD, and all the remaining algorithms are newly developed. Further, these four proposed Bregman-divergence based algorithms can work for any relative-smooth problems, relaxing the original requirement of the Lipschitz smoothness condition.   

It is deserving recognizing that the second-order convergence theory for these traditional proximal alternating minimization algorithms (e.g., PALM  \cite{attouch2010proximal,bolte2014proximal} and PAM  \cite{xu2013block}) is still an open problem in the literature. Therefore, this work not only relaxes the requirement of the Lipschitz smoothness condition but also solves an open problem of the second-order convergence guarantees to complement the current first-order convergence theories in \cite{attouch2010proximal,bolte2014proximal,xu2013block}. For convenience, we compare the proposed Bregman-divergence based algorithms with the according state-of-the-art Euclidean-distance based algorithms in \Cref{tab:compare}.

\begin{table}[h!t]
    \centering
                \begin{tabular}{lcccll}
                    \toprule 
                    \blk{0.14}{\textbf{Algorithms}} 
                    & \blk{0.2}{\textbf{Lipchitz Smoothness}}    
                    & \blk{0.2}{\textbf{First-order Convergence}}
                    &\blk{0.2}{\textbf{Second-order Convergence}}   
                    \\
                    \midrule 
                    GD\cite{lee2017first} &\cmark & \cmark  &
                    \cmark
                    \\[0.5ex]
\rowcolor[gray]{0.9}B-GD, \Cref{alg:bgd} &\xmark & \cmark  &
                    \cmark
                    \\[0.5ex]
                    PPM\cite{nocedal2006numerical,lee2017first} &\cmark&\cmark&\cmark
                    \\
\rowcolor[gray]{0.9}B-PPM, \Cref{alg:bpd} &\xmark&\cmark&\cmark
                    \\[0.5ex]
                    PALM\cite{bolte2014proximal} &\cmark&\cmark& \xmark
                    \\
\rowcolor[gray]{0.9}B-PALM, \Cref{alg:balsg} &\xmark&\cmark&\cmark
                    \\[0.5ex]
                    PAM\cite{attouch2010proximal,xu2013block} &\cmark&\cmark& \xmark
                    \\[0.5ex]
\rowcolor[gray]{0.9}B-PAM, \Cref{alg:bals} &\xmark&\cmark&\cmark
                    \\[-0.2ex]
                    \bottomrule
                \end{tabular} 
    \caption {Compare the proposed Bregman methods: B-GD, B-PPM, B-PALM, and B-PAM with several popular Euclidean-distance based algorithms: GD, PPM, PALM, and B-PAM.}
\label{tab:compare}
\end{table}

\medskip  
We build our main results upon the following assumptions from Lemmas \ref{lem:general:decrease} and \ref{lem:general:decrease:coordinate}.
 \begin{assump}
\label{assump:f}
$f\in\mathcal{C}^{2}$ is a coercive, lower-bounded,  KL function.
\end{assump}

\begin{assump}
\label{assump:h}
$f$ is relatively (bi) smooth  w.r.t. to some strongly (bi) convex and (bi) super-coercive function $h\in\calC^2$.
\end{assump}

\begin{theorem}\label{thm:main}
Under Assumptions  \ref{assump:f}--\ref{assump:h}, 
B-GD, B-PPM, B-PALM, and B-PAM converge almost surely to a second-order stationary point of $f$ from random initialization.
\end{theorem}

Some remarks are as follows. 

First of all, these assumptions are mild and we show that any finite-degree polynomial objective function or even a non-polynomial objective function (but with a polynomial-order Hessian) satisfies all these assumptions. See \Cref{sec:applications} for detailed discussions.

In addition, it is worth noting that the coercivity and KL assumptions are used merely to show the convergence to a critical point, and are not assumed in showing the avoiding-saddle property, like what Lee et al. did in the seminal work \cite{lee2017first}. Therefore, this work complements \cite{lee2017first} by obtaining the same results but without requiring the Lipschitz smoothness condition and keeping the same other assumptions. 

That being said, one may argue that with the coercivity assumption, it is always possible to compute the local Lipschitz constant---related to the initialization $\vx_0$---for the level set $\calB_{\vx_0}=\{\x: f(\x)\le f(\x_0)\}$, denoted by $L_f(\calB_{\x_0}):=\max_{\x\in\calB_{\vx_0}}\|\nabla^2 f(\x)\|$. Then, gradient descent with $\eta<1/L_f(\calB_{\x_0})$ obeys both sufficient decrease and avoiding-saddle properties. However, there are two main drawbacks to implement this idealistic approach in practice. First, each time it requires additionally resources to compute the level set $\calB_{\vx_0}$ which could be time consuming. Second, the set $L_f(\calB_{\x_0})$ could be very large, giving a very large local Lipschitz constant $L_f(\calB_{x_0})$ which then forces GD to use a very small step size, resulting in 
an extremely poor algorithm efficiency (like \Cref{fig:bgd:PCA}). Indeed, this is one important advantage of B-GD \cite{bolte2018first}  to allow adaptive step sizes. We complement this work  by providing the second-order convergence theory for B-GD.

Further, it is worth reminding that the proposed Bregman-divergence based optimization methods cover both the proximal alternating minimization (PAM) and the proximal alternating linearized minimization (PALM) when we specialize the Bregman divergence to the Euclidian distance. Therefore, this work not only develops guaranteed optimization methods for non-Lipschitz smooth problems but also solves an open problem of the second-order convergence guarantees to complement the current first-order convergence theories \cite{attouch2010proximal,bolte2014proximal,xu2013block}.
 
Finally,  as many popular (nonconvex) machine learning and signal processing problems \cite{ge2017no,sun2016geometric,chi2018nonconvex,zhu2018global, zhu2017global,li2018non, LiEtAl2017Geometry, zhu2019global} have such a landscape property as all second-order stationary points are globally optimal solutions,  the global optimality can be achieved by the proposed  Bregman-divergence based algorithms in solving this particular class of problems.
\medskip

\section{Stylized Applications}\label{sec:applications}
\subsection{Polynomial Objective Functions}
First of all, we show that any lower-bounded and coercive finite-degree polynomial function satisfies all the main assumptions. Before proceeding,  we recall that $f(\x)$ is a $d$th-degree polynomial  if the highest degree of $\x$ among all monomials of $f(\x)$ is $d$. This definition can be easily extended to the bi-variable case: we define that $f(\x,\y)$ is a $(d_1,d_2)$th-degree polynomial if it is a $d_1$th-degree polynomial when $\y$ is fixed and $d_2$th-degree polynomial when $\x$ is fixed.
\begin{lem}\label{lem:h:exist}
Suppose $f$ is any coercive and lower-bounded $d$th-degree (or $(d_1,d_2)$th-degree for the bi-variable case) polynomial function with $d,d_1,d_2\ge 2$. Set  $h$ to be
\begin{equation}
\begin{aligned}
h(\x)&=
\frac{\alpha}{d}\|\x\|_{2}^{d}+\frac{\sigma}{2}\|\x\|_{2}^{2}+1,
\\
h(\x,\y)&=
\left(\frac{\alpha}{d_1}\|\x\|_{2}^{d_1}+\frac{\sigma}{2}\|\x\|_{2}^{2}+1\right)\left(\frac{\alpha}{d_2}\|\y\|_{2}^{d_2}+\frac{\sigma}{2}\|\y\|_{2}^{2}
+1\right)
\label{eqn:h:exist}
\end{aligned}
\end{equation}
for any $\alpha,~\sigma>0$.
Then $(f,h)$ satisfies Assumptions  \ref{assump:f}--\ref{assump:h}.
\end{lem}
Lemma~\ref{lem:h:exist} is proved in \Cref{sec:pf:lem:h:exist}. Now together with \Cref{thm:main}, we obtain that the proposed Bregman algorithms can be used to minimize any lower-bounded finite-degree polynomial.
\begin{cor}\label{cor:polynomial}
Suppose $f$ is any coercive, lower-bounded $d$th-degree (or $(d_1,d_2)$-degree) polynomial function with $d,d_1,d_2\ge 2$.  Set $h$ according to \eqref{eqn:h:exist}.
Then  B-GD, B-PPM, B-PALM, and B-PAM almost surely converge to a second-order stationary point of $f$ from random initialization.
\end{cor}
Due to the requirement of the Lipschitz smoothness condition, the current theory for most traditional first-order (or even second-order) and alternating minimization algorithms cannot accommodate high-degree (larger than 2) polynomial objective functions, which sets demanding restrictions on the applications and consequently excludes most practical applications related with the matrix factorizations \cite{ge2017no, chi2018nonconvex}, which generally involve fourth-degree polynomial objective functions. \Cref{cor:polynomial} solves this problem by stating that the proposed Bregman-divergence based algorithms can be used to obtain a second-order stationary point of these problems. Further, many popular (nonconvex) machine learning and signal processing problems allow for all second-order stationary points to be globally optimal, which implies the global optimality of the proposed  Bregman-divergence based algorithms in solving a particular class of problems.

\subsection{Non-polynomial Objective Functions}

\begin{lem}\label{lem:h:exist:2}
Suppose $f\in\calC^2$ with the spectral norms of its (partial) Hessians 
\[\|\nabla^2f(\x)\|\le C_{1}+C_{2}\|\x\|_{2}^{d-2}\]
or (for the bi-variable case)
\begin{align*}
\|\nabla^{2}_{\x\x}f(\x,\y)\|&\le \left(C_{1}+C_{2}\|\x\|_{2}^{d_1-2}\right)\left(C_3+C_4\|\y\|_2^{d_2}\right),\\
\|\nabla^{2}_{\y\y}f(\x,\y)\|&\le \left(C_5+C_6\|\x\|_2^{d_1}\right)\left(C_{7}+C_{8}\|\y\|_{2}^{d_2-2}\right)
\end{align*}
for some integers  $d,d_1,d_2\ge 2$ and positive constants $C_{1}$ to $C_8$.
Then $f$ is relatively (bi) smooth w.r.t. $h$ defined in \eqref{eqn:h:exist}.
\end{lem}
The proof of \Cref{lem:h:exist:2} is arranged in \Cref{proof:lem:h:exist:2}. 
\begin{cor}
\label{cor:nonpolynomial}
Suppose $f\in\mathcal{C}^{2}$ is any coercive and lower-bounded KL function with its Hessian (or partial Hessian) spectral norms upper bounded by a polynomial as in \Cref{lem:h:exist:2}. Set  $h$ according to \eqref{eqn:h:exist}.
Then B-GD, B-PPM, B-PALM, and B-PAM almost surely converge to a second-order stationary point of $f$ from random initialization.
\end{cor}
\Cref{cor:nonpolynomial} is important in dealing with those optimization problems with non-polynomial objective functions. It provides the second-order convergence guarantees for these problems as long as the spectral norm of their Hessian matrix in a polynomial order. 

\subsection{Global Optimality in Low-rank Matrix Recovery}\label{sec:BMF}
A natural way to solve large-scale matrix optimization problems is the  {\em Burer-Monteiro factorization method} (BMF) \cite{burer2003nonlinear, burer2005local}. Given a  general rank-constrained matrix optimization problem
\begin{align}
\minimize_{\mX\in\mathbb{S}_+^n~\tmop{or}~\mX\in\R^{n\times m}} q(\mX)~~\st ~ \rank(\mX)\le r
\label{eqn:BMF0}
\end{align}
the BMF method re-parameterizes the problem by setting $\mX=\mU\mU^\top$ (for symmetric case) or $\mX=\mU\mV^\top$ (for nonsymmetric case) and then focuses on the new BMF problems:
\begin{align}\minimize_{\mU\in\R^{n\times r}}f(\mU):=q(\mU\mU^\top)\quad\tmop{and}\quad
\minimize_{\mU\in\R^{n\times r},\mV\in\R^{m\times r}}f(\mU,\mV):=q(\mU\mV^\top)
\label{eqn:BMF1}    
\end{align}
Then a direct consequence of \Cref{cor:polynomial} is the second-order convergence of the Bregman-divergence based methods in solving the BMF problems \eqref{eqn:BMF1}.
\begin{cor}\label{cor:BMF}
Assume $f$ defined in \eqref{eqn:BMF1} is a coercive, lower-bounded and finite-degree polynomial function. Set $h$ according to \eqref{eqn:h:exist}. Then  B-GD, B-PPM, B-PALM, and B-PAM are guaranteed to almost surely converge to a globally optimal solution of $f$ from random initialization.
\end{cor}
The BMF method becomes increasingly popular in recent years in solving large-size matrix optimization problems. This is not only due to its high computational efficiency, but also because of the recent breakthroughs \cite{li2018non,ge2017no} in connecting the globally optimal solutions of the original objective function $q$ in \eqref{eqn:BMF0} and the second-order stationary points of the reformulated BMF objective function $f$ in \eqref{eqn:BMF1}. It has been already proved that when the original matrix function is well-behaved (e.g., the condition number of the Hessian matrix is well-controlled when evaluated on the low-rank matrices \cite{li2018non,ge2017no}), then every second-order stationary point of the BMF objective function $f$ corresponds to a global optimal solution of the original objective function $q$.
Therefore, in this sense, the second-order convergence of the Bregman-divergence based methods means the global optimality for solving a particular class of matrix problems, including but not limited to matrix PCA, matrix sensing, matrix completion problems.

\section{Convergence Analysis}\label{app:proof:main}

\subsection{Main Ingredient for First-order Convergence Analysis}\label{sec:KL}

The KL property characterizes the local geometry of the objective function around the critical points, basically saying that the function landscape is not quite flat compared with the norm of the gradient evaluated around the critical points. Formally, it is defined as:
\begin{defi}\cite{attouch2010proximal,bolte2014proximal,bolte2018first}\label{def:KL}
    We say a proper semi-continuous function $f(\x)$ satisfies Kurdyka-Lojasiewicz (KL) property, if ${\x^\star}$ is a limiting critical point of $f(\x)$, then there exist $\delta>0,~\theta\in[0,1),~C>0,~s.t.$
    \[
    \left|f(\x) - f(\x^\star)\right|^{\theta} \leq C \|\nabla f(\vx)\|_2,~~\forall~\x\in \calB(\x^\star, \delta)
    \]
    
\end{defi}
We mention that the above KL property (also known as KL inequality) states the regularity of $h(\vu)$ around its critical point $\vu$ and the KL inequality trivially holds at non-critical points.  A function satisfying the KL property is a KL function. A very large set of functions are KL functions: as stated in \cite[Theorem 5.1]{bolte2014proximal},  for a proper lower semi-continuous function, it has KL property once it is semi-algebraic.  And the semi-algebraic property of sets and functions is sufficiently general,  including but never limited to any polynomials, any norm, quasi-norm, $\ell_0$ norm, smooth manifold, etc. For more discussions and examples, see \cite{attouch2010proximal,bolte2014proximal}.

The KL property plays a crucial role in establishing the first-order convergence (a.k.a. sequence convergence) for a number of descent type algorithms (see, e.g., \cite{attouch2010proximal,bolte2014proximal,bolte2018first}). It has been shown that as long as a generated sequence satisfies the following {\bf (C1)} sufficient decrease property and {\bf (C2)} the bounded gradient property, then this sequence is guaranteed to converge to a first-order stationary point (a.k.a. critical point).

\begin{defi}[Definition 4.1, \cite{bolte2018first}]
\label{def:grad:seqence}
Let $f:\R^{n}\to\R$ be a continuous function. A sequence $\{\x_{k}\}_{k\in\N}$ is called a gradient-like descent sequence for $f$ if the following two conditions hold
for some positive constants $\rho_1,~\rho_2$:

\begin{description}
\item\textbf{\em (C1) Sufficient decrease property}: $ f(\x_{k}) -f(\x_{k+1})\ge \rho_{1} \left\| \x_{k+1}-\x_{k} \right\|_{2}^{2},~ \forall~k\in\N$;

\item\textbf{\em (C2) Bounded gradient property}: $\left\| \nabla f(\x_{k+1}) \right\|_2 \le  \rho_{2}  \left\| \x_{k+1}-\x_{k} \right\|_{2}, ~ \forall~k\in\N$.
\end{description}
\end{defi}

\begin{theorem}[Theorem 6.2, \cite{bolte2018first}]\label{thm:KL}
Let $f:\R^{n}\to\R$ be any continuous KL function and $\nabla$ be the gradient operator. Let $\{\x_{k}\}_{k\in\N}$ be a bounded gradient-like descent sequence for $f$. Then the sequence $\{\x_{k}\}_{k\in\N}$ converges to a critical point of $f$.
\end{theorem}

\subsection{Main Ingredient for Second-order Convergence Analysis}

\begin{defi} Let $f$ be a twice continuously differentiable function. Then
\begin{enumerate}
\item $\x$ is a first-order stationary point (a.k.a. critical point) of  $f$ if $\nabla f(\x)=\vzero$;
\item $\x$ is a second-order stationary point of  $f$ if $\nabla f(\x)=\vzero$ and $\nabla^{2}f(\x)\succeq 0$;
\item  $\x$ is a strict saddle of $f$ if it is a first-order but not a second-order stationary point.
\end{enumerate}
\label{def:critical:point}
\end{defi} 

\begin{defi}[Unstable Fixed Points] \label{def:unstable:fixed:point} 
Let $g:\R^n\to\R^n$ be a continuously differentiable mapping.
Then the unstable fixed point of $g$ is defined as any fixed point of $g$ with $\tmop{Sp}(Dg(\x^\star))>1$, where $\tmop{Sp}(\cdot)$ denotes the spectral radius  (i.e., the largest magnitude eigenvalue) and $D$ denotes the Jacobian operator.
\end{defi}
The seminal work \cite{lee2017first} establishes that certain iterative algorithms can avoid strict saddle points by viewing the iterative algorithm as a dynamic system and proving that any strict saddle point of the objective function is an unstable fixed point of the dynamic system. Then by the stable manifold theorem \cite{shub2013global},  the event for this algorithm (with a random initialization) to converge to a strict saddle has a zero probability. This is summarized in the following result.
\begin{theorem}[Theorem 2, \cite{lee2017first}]\label{thm:jason}
Let $g:\R^n\to\R^n$ be a continuously differentiable mapping.
Suppose $\det(D g(\x))\neq 0$ in the entire domain. Then the set of initial points that converge to the unstable fixed points of $g$ is of zero Lebesgue-measure.
\end{theorem}
This implies that as long as the algorithm determined by $g$ uses a random initialization and converges to a critical point,  then this critical point would be a second-order stationary point of $f$.

\subsection{Convergence Analysis of B-GD}

\subsubsection{First-order Convergence of B-GD}
\begin{theorem}
Under Assumptions  \ref{assump:f}--\ref{assump:h}, B-GD  converges to a critical point of $f$.
\end{theorem}

\begin{proof}
First,  B-GD is well-defined in view of \Cref{lem:general:decrease}.
Then in view of \Cref{thm:KL} and the assumption that $f$ is KL function, it is sufficient to prove that $\{\x_{k}\}_{k\in\N}$ is a gradient-like descent sequence for $f$, i.e., to show conditions \textbf{(C1)} and \textbf{(C2)}.
Condition \textbf{(C1)} follows from \eqref{eqn:general:decrease1} in \Cref{lem:general:decrease}. Condition \textbf{(C2)}  holds because by the optimality condition
\begin{equation}
\begin{aligned}
&\nabla f(\x_{k})+
({\nabla h(\x_{k+1})-\nabla h(\x_{k})})/{\eta}=\zero\label{opt:bgd:1},
\\
\Rightarrow&\|\nabla f(\x_{k})\|_{2}=\frac{1}{\eta}\|\nabla h(\x_{k+1})-\nabla h(\x_{k})\|_{2}
\le \frac{\rho_{h}(\calB_0)}{\eta}\|\x_{k+1}-\x_{k}\|_{2},
\end{aligned}
\end{equation} 
where the inequality directly follows by combining the sufficient decrease property of each iteration, the coercivity of $f$, and the twice differentiability of $h$.\footnote{
First, the sufficient decrease property of each iteration  ensures all iterates $\{\x_k\}$ live in the sub-level set $\calB_0:=\{\x: f(\x)\le f(\x_0)\}$, which is guaranteed to be  a bounded set by the coercivity of $f$ (cf. \cite[Prop. 11.11]{bauschke2011convex}); Second, given any twicely continuous function $h\in\calC^2$ (or $f\in\calC^2$),  $\|\nabla^2h\|$ (or $\|\nabla^2f\|$) is a continuous function and must have a maximum over the closure of $\calB_0$, for convenience denoted by $\rho_h(\calB_0)$ (or $\rho_f(\calB_0)$).\label{footnote:local}}
Similarly,   
\begin{equation}
\begin{aligned}
&\|\nabla f(\x_{k+1})\|_{2}=\|\nabla f(\x_{k})-\nabla f(\x_{k})+\nabla f(\x_{k+1})\|_{2}
\\
& \le \|\nabla f(\x_{k})\|_2+\|\nabla f(\x_{k})-\nabla f(\x_{k+1})\|_{2} \le( \frac{\rho_{h}(\calB_0)}{\eta} +\rho_{f}(\calB_0) )\|\x_{k+1}-\x_{k}\|_{2},    
\label{eqn:footnote}
\end{aligned}
\end{equation}
 where the last inequality follows by \eqref{opt:bgd:1} and \Cref{footnote:local}.
\end{proof}

\subsubsection{Second-order Convergence of B-GD}
\begin{theorem}  Under Assumptions  \ref{assump:f}--\ref{assump:h}, B-GD with random initialization almost surely converges to a second-order stationary point of $f$.
\end{theorem}

\begin{proof}
As we have proved the first-order convergence, to show the second-order convergence from the first-order convergence, it suffices to use \Cref{thm:jason} to show that B-GD  avoids strict saddles. 
For that purpose, we define \eqref{eqn:bgd1} as $\x_{k}=g(\x_{k-1})$ and compute the Jacobian $Dg$. By the definition of $g$, we get \[Dg(\x_{k})=\partial\x_{k+1}/\partial\x_{k}^\top.\]
 Then
we apply the implicit function theorem to the optimality condition \eqref{opt:bgd:1} and in view of  the non-singularity of $\nabla^2 h$, we obtain that $Dg$ is continuous and given by
\[D g(\x_{k})=\left[\nabla^2 h(\x_{k+1}) \right]^{-1}(\nabla^2 h(\x_{k})-\eta\nabla^2 f(\x_{k})).
\]
Since the above analysis holds for all $\x_{k}\in\R^n$, this further implies that $Dg(\x)$ is continuous and given by
\begin{align}
D g(\x)=[\nabla^2 h(g(\x)) ]^{-1}(\nabla^2 h(\x)-\eta\nabla^2 f(\x)).
\label{eqn:Dg:bgd:1}
\end{align}
To show the avoidance of strict saddles, by \Cref{thm:jason}, it suffices to show: 
\smallskip\paragraph{\bf (1) Showing $g$ is a continuously differentiable mapping}
This follows from the continuity of $Dg$ in \eqref{eqn:Dg:bgd:1}.

\smallskip\paragraph{\bf (2) Showing $\det(D g)\neq 0$ in the whole domain} This directly follows from the expression of $Dg(\x)$ \eqref{eqn:Dg:bgd:1}, and along with the positive definiteness of $\nabla^{2} h$ and $\nabla^{2} h \pm{\eta} \nabla^{2}f$.

\smallskip\paragraph{\bf (3) Showing any strict saddle  is an unstable fixed point}
For any strict saddle $\x^\star$, it is a fixed point, i.e., $g(\x^\star)=\x^\star$.
Plugging this into \eqref{eqn:Dg:bgd:1} gives
\begin{align*}
D g(\vx^\star)
=&[\nabla^2 h(\x^\star) ]^{-1}(\nabla^2 h(\x^\star)-\eta\nabla^2 f(\x^\star))
\\
\sim &
[\nabla^2 h(\vx^\star) ]^{-\frac{1}{2}}(\nabla^2 h(\vx^\star)-\eta\nabla^2 f(\vx^\star))[\nabla^2 h(\vx^\star) ]^{-\frac{1}{2}}
\\
=&\eye-\eta [\nabla^2 h(\vx^\star) ]^{-\frac{1}{2}}\nabla^2 f(\x^\star)[\nabla^2 h(\vx^\star) ]^{-\frac{1}{2}}
\\
:=&\eye-\eta\mPhi
\end{align*}
with ``$\sim$" denotes the matrix similarity. Therefore, 
\[\Sp(Dg(\x^\star))=\Sp(\eye-\eta\mPhi)>1-\eta \min_{i}\lambda_i(\mPhi)>1,\]
 since
$\mPhi$ is congruent to $\nabla^2f(\x^\star)$, which has at least a negative eigenvalue at strict saddles.
\end{proof}

\subsection{Convergence Analysis of B-PPM}
\subsubsection{First-order Convergence of B-PPM}
\begin{theorem}
Under Assumptions  \ref{assump:f}--\ref{assump:h}, B-PPM  converges to a critical point of $f$.
\end{theorem}

\begin{proof}
First of all, B-PPM  is well-defined by \Cref{lem:general:decrease}.
Then, by \Cref{thm:KL} and the assumption that $f$ is a KL function, it is sufficient to prove that $\{\x_{k}\}_{k\in\N}$ is a gradient-like descent sequence for $f$, i.e., showing conditions \textbf{(C1)} and \textbf{(C2)}.
Condition \textbf{(C1)} follows from \eqref{eqn:general:decrease2} in \Cref{lem:general:decrease}. Condition \textbf{(C2)}  holds because by the optimality condition
\begin{align}
\nabla f(\x_{k+1})+
({\nabla h(\x_{k+1})-\nabla h(\x_{k})})/{\eta}=\zero\label{opt:bpd:1}
\end{align}
$    \|\nabla f(\x_{k+1})\|_{2}=\frac{1}{\eta}\|\nabla h(\x_{k+1})-\nabla h(\x_{k})\|_{2}
\le \frac{\rho_{h}(\calB_0)}{\eta}\|\x_{k+1}-\x_{k}\|_{2},
$    where the inequality follows from \Cref{footnote:local}.
\end{proof}

\subsubsection{Second-order Convergence of B-PPM}
\begin{theorem}
Under Assumptions  \ref{assump:f}--\ref{assump:h}, B-PPM converges almost surely to a second-order stationary point of $f$ from random initialization.
\end{theorem}

\begin{proof}
As we have already shown that B-PPM converges to a first-order critical point, it remains to use \Cref{thm:jason} to show that this first-order critical point will not be a strict saddle for almost sure, and hence would be a second-order stationary point. To apply \Cref{thm:jason}, we define
\eqref{eqn:bpd1} as $\x_{k}=g(\x_{k-1})$ and compute the Jacobian matrix $Dg$.  By the definition of $g$, we have 
\[D g(\x_{k})=\partial\x_{k+1}/\partial\x_{k}^\top.\] Now we apply the implicit function theorem to \eqref{opt:bpd:1} and in view of the non-singularity of $\nabla^2 h+\eta\nabla^2f$, we obtain that $Dg$ is continuous and given by
\[ D g(\x_{k})=\left(\nabla^2 h(\x_{k+1}) +\eta\nabla^{2}f(\x_{k+1})\right)^{-1}\nabla^2 h(\x_{k}).
\]
Noting that the above argument holds for any $\x_{k}\in\R^n$, therefore,  $D g(\x)$ is continuous and given by
\begin{align}
D g(\x)=\left(\nabla^2 h(g(\x))+\eta\nabla^{2}f(g(\x)) \right)^{-1}\nabla^2 h(\x).
\label{eqn:Dg:bpd:1}
\end{align}
By \Cref{thm:jason}, the remaining part of the proof consists of showing the following conditions.

\smallskip\paragraph{\bf (1) Showing  $g$ is a continuously differentiable mapping}
This immediately follows from the continuity of $Dg$ in \eqref{eqn:Dg:bpd:1}.

\smallskip\paragraph{\bf (2) Showing $\det(D g)\neq 0$} This is because by the expression of $Dg$:
$\det(D g(\x))=\det([\nabla^2 h(g(\x))+\eta\nabla^{2}f(g(\x)) ]^{-1})\det(\nabla^2 h(\x))$ and both $\nabla^{2} h$ and $\nabla^{2} h \pm{\eta} \nabla^{2}f$ are positive definiteness for any $\eta<1/L_f$.

\smallskip\paragraph{\bf (3) Showing any strict saddle is an unstable fixed point}
First for any strict saddle $\x^\star$, we have $\x_{k+1}=\x_{k}$ when $\x_k=\x^\star$, indicating that $\x^\star$ is a fixed point, i.e., $g(\x^\star)=\x^\star$. Now plugging  $g(\x^\star)=\x^\star$ to \eqref{eqn:Dg:bpd:1}, we have
\begin{align*}
D g(\vx^\star)=&\left[\nabla^2 h(\x^\star)+\eta\nabla^{2}f(\x^\star) \right]^{-1}\nabla^2 h(\x^\star)
\\
\sim& [\nabla^2 h(\vx^\star) +\eta \nabla^{2}f(\x^{\star})]^{-1/2}(\nabla^2 h(\vx^\star))[\nabla^2 h(\vx^\star) +\eta \nabla^{2}f(\x^{\star})]^{-1/2}\\
=& \eye\!-\!\eta[\nabla^2 h(\vx^\star) \!+\!\eta \nabla^{2}f(\x^{\star})]^{-1/2}\nabla^2 f(\x^\star)[\nabla^2 h(\vx^\star) \!+\!\eta \nabla^{2}f(\x^{\star})]^{-1/2}
\\
:=&\eye\!-\!\eta\mPhi.
\end{align*}
Clearly, we know $D g(\vx^\star)$  has an eigenvalue strictly larger than 1 since $\mPhi$ is congruent to  $\nabla^{2}f(\x^{\star})$, which has a negative eigenvalue.
\end{proof}

\subsection{Convergence Analysis of  B-PALM} \label{proof:thm:balsg:1}

\subsubsection{First-order Convergence of B-PALM}

\begin{theorem}\label{thm:balsg:1}
Under Assumptions  \ref{assump:f}--\ref{assump:h}, B-PALM  with arbitrary initialization converges to a critical point of $f$.
\end{theorem}

\begin{proof}
First of all, in view of \Cref{lem:general:decrease:coordinate}, we immediately conclude that  B-PALM is well-defined. Now, by \Cref{thm:KL} and the assumption that $f$ is KL function, it is sufficient to prove that $\{(\x_{k},\y_{k})\}_{k\in\N}$ is a gradient-like descent sequence for $f$, i.e., showing conditions \textbf{(C1)} and \textbf{(C2)}.
Condition \textbf{(C1)}  directly follows from \Cref{lem:general:decrease:coordinate}.

Now, we show condition \textbf{(C2)}. To simplify notations in the proof, we rewrite the iteration of B-PALM \eqref{update:gd:coordinate} as  
\begin{equation}
\begin{aligned}
\x_{+}=& \argmin_{\x'} \lg \nabla_{\x} f(\x',\y), \x'-\x\rg+\frac{1}{\eta}D^1_{h}(\x',\x;\y),\\
\y_+=& \argmin_{\y'} \lg\nabla_{\y'} f(\x_{+},\y),\y'-\y\rg+\frac{1}{\eta}D^2_{h}(\y',\y;\vx_{+}),
\label{eqn:new:BPALM}
\end{aligned}
\end{equation}
where the optimality condition for the first-block is given by
\begin{align}
\nabla_{\x} h(\x_{+},\y)=\nabla_{\x} h(\x,\y)-\eta \nabla_{\x}f(\x,\y)\label{eqn:first:balsg}
\end{align}
which then implies that
\[\|\nabla_{\x} f(\x,\y)\|_{2}=\frac{1}{\eta}\|\nabla_{\x} h(\x_{+},\y)-\nabla_{\x} h(\x,\y)\|_{2}  \leq   \frac{\rho_h(\calB_0)}{\eta} \| (\x_{+},\y_{+})-(\x,\y) \|_{2},
\]
where the inequality follows  by \Cref{footnote:local} and $\|(\x_{+},\y)-(\x,\y)\|_2\le\|(\x_{+},\y_{+})-(\x,\y)\|_2$.
Then applying a similar analysis to the  
optimality condition of the second-block of B-PALM:
\begin{align}
\nabla_{\y} h(\x_+, \y_{+})=\nabla_{\y} h(\x_+,\y)-\eta \nabla_{\y}f(\x_{+},\y),
\label{eqn:second:balsg}
\end{align}
we get
\[
\|\nabla_{\y} f(\x_{+},\y)\|_{2} \leq \frac{\rho_h(\calB_0)}{\eta}\|\y_{+}-\y\|_{2}.
\]
Using a similar argument as in \eqref{eqn:footnote}, we have
\begin{align*}
\|\nabla_{\y} f(\x,\y)\|_{2}
\leq \left(  \frac{\rho_h(\calB_0)}{\eta} +\rho_f(\calB_0)  \right) \| (\x_{+},\y_{+})-(\x,\y) \|_{2}.
\end{align*}
Combining the above two, we get an equivalent version of the bounded gradient property
\begin{align*}
\|\nabla f(\x,\y)\|_{2}
\le& \|\nabla_{\x} f(\x,\y\|_{2} \!+\!\|\nabla_{\y} f(\x,\y)\|_{2}
\le \!(  \frac{2\rho_h(\calB_0)}{\eta} \!+\!\rho_f(\calB_0)  \| (\x_{+},\y_{+})\!-\!(\x,\y) \|_{2}.
\end{align*}
Therefore,
$
\|\nabla f(\x_{+},\y_{+})\|_{2}
\le \left(  \frac{2\rho_h(\calB_0)}{\eta} +2\rho_f(\calB_0)  \right)\| (\x_{+},\y_{+})-(\x,\y) \|_{2}.
$
\end{proof}

\subsubsection{Second-order Convergence of  B-PALM} \label{proof:thm:balsg:2}
\begin{theorem}\label{thm:balsg:2}     Under Assumptions  \ref{assump:f}--\ref{assump:h}, B-PALM converges almost surely to a second-order stationary point of $f$ from random initialization.
\end{theorem}

\begin{proof}
As we have shown the first-order convergence, to show the second-order convergence, we will use \Cref{thm:jason} to show that the B-PALM avoids strict saddle for almost surely. We start by computing the algorithmic mapping $g$ of the B-PALM. Towards that end, 
we rewrite  B-PALM \eqref{eqn:new:BPALM} as
\begin{align*}
(\x_{+},\y)&=g_{1}(\x,\y),
\\
(\x,\y_{+})&=g_{2}(\x,\y).
\end{align*}
The mappings $g_{1},~g_{2}$ are well-defined in the whole domain $\R^{n}\times \R^{m}$  in view of strong convexity and coercivity of the objective function in \eqref{eqn:new:BPALM}.
Then  B-PALM  can be viewed as iteratively performing the following composite mapping for $k=1,2,\ldots$
\begin{align}
(\x_{k},\y_{k})=g(\x_{k-1},\y_{k-1}),
\label{eqn:g:balsg}
\end{align}
where the mapping $g$ is defined as the composition $g:=g_{2}\circ g_{1}
$. Therefore, we can use the chain rule to compute $Dg$. For this purpose, let us first compute  $Dg_{1}$ and $Dg_{2}$, respectively. We compute $Dg_1$ in view of the definition $g_1$ $(\x_{+},\y)=g_{1}(\x,\y)$ and obtain that
\[
Dg_{1}(\x,\y)=\begin{bmatrix} {\partial \x_{+}}/{\partial \x^{\top}} & {\partial \x_{+}}/{\partial \y^{\top}} \\
\zero & \eye_{m}  \end{bmatrix}.
\]
To compute ${\partial \x_{+}}/{\partial \x^{\top}}$ and ${\partial \x_{+}}/{\partial \y^{\top}}$, we can apply the implicit function theorem to the optimality condition \eqref{eqn:first:balsg} to get that
\begin{align*}
\nabla_{\x\x}^2 h(\x_{+},\y)({\partial \x_{+}}/{\partial \x^{\top}})
&=\nabla_{\x\x}^2 h(\x,\y)-\eta\nabla^2_{\x\x}f(\x,\y),
\\
\nabla_{\x\x}^2 h(\x_{+},\y)({\partial \x_{+}}/{\partial \y^{\top}})
&=\nabla_{\x\y}^2 h(\x,\y)-\nabla_{\x\y}^2 h(\x_{+},\y)-\eta\nabla^2_{\x\y}f(\x,\y).
\end{align*}
Then in view of the strong bi-convexity of  $h$, we can further get
\begin{align*}
{\partial \x_{+}}/{\partial \x^{\top}}
&=\nabla^{2}_{\x\x}h(\x_{+},\y)^{-1}\left(\nabla^{2}_{\x\x}h(\x,\y)-\eta\nabla^2_{\x\x}f(\x,\y)\right),
\\
{\partial \x_{+}}/{\partial \y^{\top}}
&=\nabla^{2}_{\x\x}h(\x_{+},\y)^{-1}(\nabla_{\x\y}^2 h(\x,\y)-\nabla_{\x\y}^2 h(\x_{+},\y)-\eta\nabla^2_{\x\y}f(\x,\y)).
\end{align*}
Similarly, the implicit function theorem can be applied to the optimality condition \eqref{eqn:second:balsg} to compute ${\partial \y_{+}}/{\partial \x^{\top}}$ and ${\partial \y_{+}}/{\partial \y^{\top}}$.  As a result, we have
\begin{equation}
\begin{aligned}
&Dg_{1}(\x,\y)
=
\begin{bmatrix}\nabla^{2}_{\x\x}h(\x_{+},\y)^{-1}
 & \zero
\\
\zero& \eye_{m}
\end{bmatrix}
\\
&
\begin{bmatrix}
\nabla^{2}_{\x\x}h(\x,\y)-\eta\nabla^2_{\x\x}f(\x,\y)
&
\nabla_{\x\y}^2 h(\x,\y)-\nabla_{\x\y}^2 h(\x_{+},\y)-\eta\nabla^2_{\x\y}f(\x,\y)\\
\zero & \eye_{m}
\end{bmatrix},
\\
&Dg_{2}(\x,\y)
=
\begin{bmatrix}
\eye_n&\zero\\
\zero& \nabla^{2}_{\y\y}h(\x,\y_{+})^{-1}
\end{bmatrix}
\\
&
\begin{bmatrix}
\eye_{n} & \zero\\
\nabla_{\y\x}^2 h(\x,\y)-\nabla_{\y\x}^2 h(\x,\y_{+})-\eta\nabla^2_{\y\x}f(\x,\y)
&
\nabla^{2}_{\y\y}h(\x,\y)-\eta\nabla^2_{\y\y}f(\x,\y)
\end{bmatrix}.
\label{Dg:alsg}
\end{aligned}
\end{equation}
Finally, $Dg$ is given by the following chain rule:
\begin{align}
Dg(\x,\y)=Dg_{2}(g_{1}(\x,\y))Dg_{1}(\x,\y).
\label{def:Dg:balsg}
\end{align}

By \Cref{thm:jason}, to show that the mapping $g$ can almost surely avoid the strict saddles,  it suffices to show the following conditions:

\smallskip\paragraph{\bf (1) Showing  $g$ is continuously differentiable mapping}
This follows from the continuity of $Dg$ in \eqref{def:Dg:balsg}.

\smallskip\paragraph{\bf (2) Showing $\det(D g)\neq 0$ in the whole domain}
First, $Dg=Dg_{2}Dg_{1}$ by the chain rule with the square matrices $Dg_{1},~ Dg_{2}$ given in \eqref{Dg:alsg}. Second, $Dg_{1}$ is nonsingular because of its upper-triangular block structure by \eqref{Dg:alsg}) and the positive definiteness of $\nabla_{\x\x}^2h-\eta\nabla^2_{\x\x}f$. Similarly, $Dg_2$ is also nonsingular in the entire domain.  This completes the proof.

\smallskip\paragraph{\bf (3) Showing any strict saddle is an unstable fixed point} We first show any strict saddle $(\x^{\star},\y^{\star})$ of $f$ is a fixed point of $g$. This is because $(\x_+,\y_+)=(\x^\star,\y^\star),(\x,\y)=(\x^\star,\y^\star)$ satisfy the optimality conditions of both $g_1$ and $g_2$ (cf. \eqref{eqn:first:balsg} and \eqref{eqn:second:balsg}) and $g=g_2\circ g_1$ is well-defined by \Cref{lem:general:decrease:coordinate}.

It remains to show   $\Sp(D g(\x^{\star},\y^{\star}))>1$. For convenience, denote
\begin{align}
\begin{bmatrix}\mH_1 \\ \mH_2\end{bmatrix}&:=
\begin{bmatrix} \nabla^2_{\x\x} h(\x^\star,\y^\star)
\\
\nabla^2_{\y\y} h(\x^\star,\y^\star)
\end{bmatrix}
,~~~~
\begin{bmatrix}
\mF_{11}&\mF_{12}\\ \mF_{21}&\mF_{22}
\end{bmatrix}
:=
\begin{bmatrix}
\nabla_{\x\x}^2f(\x^\star,\y^\star)&\nabla_{\x\y}^2f(\x^\star,\y^\star)\\
\nabla_{\y\x}^2f(\x^\star,\y^\star)&\nabla_{\y\y}^2f(\x^\star,\y^\star)
\end{bmatrix}
\label{def:f}
\end{align}
Now let us compute $D g(\x^{\star},\y^{\star})$ by plugging
$(\x_{+},\y_{+})=(\x,\y)=(\x^\star,\y^\star)$
in \eqref{def:Dg:balsg}:
\begin{equation*}
\begin{aligned}
D g(\x^{\star},\y^{\star})
=&
\begin{bmatrix}
\eye_{n} & \zero\\
-\eta\mH_{2}^{-1}\mF_{21}
&
\eye_{m}- \eta \mH_{2}^{-1}\mF_{22}
\end{bmatrix}
\begin{bmatrix} \eye_{n}-\eta\mH_{1}^{-1}\mF_{11} &
-\eta\mH_{1}^{-1}\mF_{12}
\\
\zero & \eye_{m}
\end{bmatrix}
\\
=&
\begin{pmatrix}
\eye \!-\! \eta\begin{bmatrix}\zero  &\\ & \mH_2^{-1}\end{bmatrix}
\begin{bmatrix}
\mF_{11}&\mF_{12}\\ \mF_{21}&\mF_{22}
\end{bmatrix}
\end{pmatrix}
\begin{pmatrix}
\eye- \eta
\begin{bmatrix}
\mH_{1}^{-1}  &\\ & \zero
\end{bmatrix}
\begin{bmatrix}
\mF_{11}&\mF_{12}\\ \mF_{21}&\mF_{22}
\end{bmatrix}
\end{pmatrix}
\\
=& \eye-
\begin{bmatrix}
\eta\mH_{1}^{-1}  &
\\
-\eta^{2} \mH_{2}^{-1}\mF_{21}\mH_{1}^{-1}& \eta\mH_{2}^{-1}
\end{bmatrix}
\begin{bmatrix}
\mF_{11}&\mF_{12}\\ \mF_{21}&\mF_{22}
\end{bmatrix}
\\
=&
\eye-\begin{bmatrix}\frac{1}{\eta}\mH_1  &
\\
\mF_{21}& \frac{1}{\eta}\mH_2\end{bmatrix}^{-1}
\begin{bmatrix}
\mF_{11}&\mF_{12}\\ \mF_{21}&\mF_{22}
\end{bmatrix}
\\
=&
\begin{bmatrix}\frac{1}{\eta}\mH_1  &
\\
\mF_{21}& \frac{1}{\eta}\mH_2\end{bmatrix}^{-1}
\begin{bmatrix}\frac{1}{\eta}\mH_1 -\mF_{11} & -\mF_{12}
\\
& \frac{1}{\eta}\mH_2-\mF_{22}\end{bmatrix}.
\end{aligned}
\end{equation*}
Because $\mH_1/\eta-\mF_{11}$ and $\mH_2/\eta-\mF_{22}$ are nonsingular, we have
\[D g(\x^{\star},\y^{\star})^{-1}=\begin{bmatrix}\frac{1}{\eta}\mH_1 -\mF_{11} & -\mF_{12}\\
& \frac{1}{\eta}\mH_2-\mF_{22}\end{bmatrix}^{-1}\begin{bmatrix}\frac{1}{\eta}\mH_1  &
\\
\mF_{21}& \frac{1}{\eta}\mH_2\end{bmatrix}.\]
To show  $\Sp(Dg(\x^\star,\y^\star))>1$, it suffices to show that $\det(Dg(\x^{\star},\y^{\star})^{-1}-\mu\eye)=0,~\tmop{for}~\mu\in (0,1)$, which is equivalent to
\begin{align*}
\iff& \det\left(\begin{bmatrix}\frac{1}{\eta}\mH_1 -\mF_{11} & -\mF_{12}
\\
& \frac{1}{\eta}\mH_2-\mF_{22}\end{bmatrix}^{-1}\begin{bmatrix}\frac{1}{\eta}\mH_1  &
\\
\mF_{21}& \frac{1}{\eta}\mH_2\end{bmatrix}-\mu\eye\right)=0
\\
\iff &
\det\left(\begin{bmatrix}\frac{1}{\eta}\mH_1  &
\\
\mF_{21}& \frac{1}{\eta}\mH_2\end{bmatrix}-\mu\begin{bmatrix}\frac{1}{\eta}\mH_1 -\mF_{11} & -\mF_{12}
\\
& \frac{1}{\eta}\mH_2-\mF_{22}\end{bmatrix}\right)=0
\\
\iff &
\det\left(\begin{bmatrix}\frac{1}{\eta}(1-\mu)\mH_1 +\mu\mF_{11} & \mu \mF_{12}
\\
\mF_{21}&
\frac{1}{\eta}(1-\mu)\mH_2 + \mu\mF_{22} \end{bmatrix}\right)=0
%\\
%\iff &
%\det\left(
%\begin{bmatrix}
%\sqrt{\mu}\eye_{n} &
%\\
%& \eye_{m}
%\end{bmatrix}
%\begin{bmatrix}\frac{1}{\eta}(1-\mu)\mH_1 +\mu\mF_{11} & \sqrt{\mu} \mF_{12}
%\\
%\sqrt{\mu}\mF_{21}&
%\frac{1}{\eta}(1-\mu)\mH_2 + \mu\mF_{22} \end{bmatrix}
%\begin{bmatrix}
%\frac{1}{\sqrt{\mu}}\eye_{n} &
%\\
%& \eye_{m}
%\end{bmatrix}
%\right)=0
\\
\iff &
\det\left(
\begin{bmatrix}
\sqrt{\mu}\eye_{n} \!\!&
\\
\!\!& \eye_{m}
\end{bmatrix}
\begin{bmatrix}\frac{1}{\eta}(1\!-\!\mu)\mH_1 \!+\!\mu\mF_{11} & \sqrt{\mu} \mF_{12}
\\
\sqrt{\mu}\mF_{21}&
\frac{1}{\eta}(1\!-\!\mu)\mH_2 \!+\! \mu\mF_{22} \end{bmatrix}
\begin{bmatrix}
\frac{1}{\sqrt{\mu}}\eye_{n} \!\!&
\\
\!\!& \eye_{m}
\end{bmatrix}
\right)=0
\\
\iff &
\det\left(
\begin{bmatrix}
\frac{1}{\eta}(1-\mu)\mH_1 +\mu\mF_{11} & \sqrt{\mu} \mF_{12}
\\
\sqrt{\mu}\mF_{21}&
\frac{1}{\eta}(1-\mu)\mH_2 + \mu\mF_{22} 
\end{bmatrix}
\right)=0
\end{align*}
Now the problem reduces to show the matrix
\[\mJ(\mu):=
\begin{bmatrix}\frac{1}{\eta}(1-\mu)\mH_1 +\mu\mF_{11} & \sqrt{\mu} \mF_{12}
\\
\sqrt{\mu}\mF_{21}&
\frac{1}{\eta}(1-\mu)\mH_2 + \mu\mF_{22} \end{bmatrix}
\]
is a singular matrix for some $\mu\in(0,1).$ First of all, observing that $\mJ(\mu)$ is symmetric and continuous (w.r.t. $\mu$), we have all its eigenvalues are real-valued an continuous (w.r.t. $\mu$) by \cite[Theorem 5.1]{kato2013perturbation}. Second,  note that
\begin{align*}
\lim\limits_{\mu\to 0^{+}}\mJ(\mu)=&
\begin{bmatrix}
\frac{1}{\eta} \mH_1 &\\
& \frac{1}{\eta} \mH_2
\end{bmatrix}\succ 0
,\quad
\mJ(1)=\begin{bmatrix}
\mF_{11}&\mF_{12}\\ \mF_{21}&\mF_{22}
\end{bmatrix}
=\nabla^{2}f(\x^{\star},\y^{\star}).
\end{align*}
First, since  $\mJ(0^{+})$ is positive definite, $\lambda_{\min}(\mJ(0^{+}))>0.$ Second, since $(\x^{\star},\y^{\star})$ is a strict saddle of $f$,  $\lambda_{\min}(\mJ(1))<0.$ Then along with that $\lambda_{\min}(\mJ(\mu))$ is continuous and real-valued, we show that $\lambda_{\min}(\mJ(\mu))=0$ for some $\mu\in (0,1)$, which completes the proof.
\end{proof}

\subsection{Convergence Analysis of B-PAM}\label{proof:thm:bals:1}

\subsubsection{First-order Convergence of B-PAM}
\begin{theorem}\label{thm:bals:1}
Under Assumptions  \ref{assump:f}--\ref{assump:h},  B-PAM with arbitrary initialization converges to a critical point of $f$.
\end{theorem}

\begin{proof}
First of all, as a direct consequence of \Cref{lem:general:decrease:coordinate}, we are guaranteed B-PAM is well defined.
Now, in view of \Cref{thm:KL} and the assumption that $f$ is KL function, it is sufficient to prove that $\{(\x_{k},\y_{k})\}_{k\in\N}$ is a gradient-like descent sequence for $f$, i.e., showing conditions \textbf{(C1)} and \textbf{(C2)}.
Condition \textbf{(C1)} directly follows from \Cref{lem:general:decrease:coordinate}.

Now, we show condition \textbf{(C2)}. To simplify notations in the proof, we rewrite the iteration of B-PALM \eqref{update:gd:coordinate} as  
\begin{equation}
\begin{aligned}
\x_{+}=& \argmin_{\x'} f(\x',\y)+\frac{1}{\eta}D^1_{h}(\x',\x;\y),\\
\y_{+}=& \argmin_{\y'}  f(\x_+,\y')+\frac{1}{\eta}D^2_{h}(\y',\y;\vx_+),
\end{aligned}
\label{eqn:new:BPAM}
\end{equation}
where the optimality condition of the first block is given by
\begin{align}
\nabla_{\x} h(\x_{+},\y)=\nabla_{\x} h(\x,\y)-\eta \nabla_{\x}f(\x_{+},\y).\label{eqn:first:bals}
\end{align}
Then as consequence of \Cref{footnote:local}, we have
\begin{align}\|\nabla_{\x} f(\x_{+},\y)\|_{2}=\frac{1}{\eta}\|\nabla_{\x} h(\x_{+},\y)-\nabla_{\x} h(\x,\y)\|_{2}\leq \frac{\rho_h(\calB_0)}{\eta}\|\x_{+}-\x\|_{2}.
\label{B-PAM:1}
\end{align}    
Then using  \eqref{B-PAM:1} and along with a similar argument as in \Cref{footnote:local}, we have
\begin{align*}
\|\nabla_{\x} f(\x_{+},\y_{+})\|_{2}& \leq \|\nabla_{\x} f(\x_{+},\y_{+})-\nabla_{\x} f(\x_{+},\y)\|_{2}+
\|\nabla_{\x} f(\x_{+},\y)\|_{2}
\\
&\leq \left( \rho_f(\calB_0)+\frac{\rho_h(\calB_0)}{\eta}\right) \| (\x_{+},\y_{+})-(\x,\y) \|_{2}.
\end{align*}
Now using a similar argument on the optimality condition for the second-block of B-PAM:
\begin{align}
\nabla_{\y} h(\x_+,\y_{+})=\nabla_{\y} h(\x_+,\y)-\eta \nabla_{\y}f(\x_{+},\y_{+}),    
\label{B-PAM:2}
\end{align}
we get that
\[
\|\nabla_{\y} f(\x_{+},\y_{+})\|_{2} \leq \frac{\rho_h(\calB_0)}{\eta}\|\y_{+}-\y\|_{2}.
\]
Therefore, by $\|\nabla f(\x_{+},\y_{+})\|_{2}
\le\|\nabla_{\x} f(\x_{+},\y_{+})\|_{2} +\|\nabla_{\y} f(\x_{+},\y_{+})\|_{2}$, we have
\begin{align*}
\|\nabla f(\x_{+},\y_{+})\|_{2}
\le  \left(  \frac{2\rho_h(\calB_0)}{\eta} +\rho_f(\calB_0)  \right)\| (\x_{+},\y_{+})-(\x,\y) \|_{2}.
\end{align*}
\end{proof}

\subsubsection{Second-order Convergence of B-PAM}
\label{proof:thm:bals:2}

\begin{theorem}\label{thm:bals:2}
Under Assumptions  \Cref{assump:f, assump:h},  B-PAM   almost surely converges to a second-order stationary point of $f$ with random initialization.
\end{theorem}

\begin{proof}
As we have already shown that B-PAM converges to a first-order critical point, it remains to use \Cref{thm:jason} to show that this first-order critical point will not be a strict saddle for almost sure. To apply \Cref{thm:jason}, we rewrite  B-PAM   \eqref{eqn:new:BPAM} as
\begin{equation}
\begin{aligned}
(\x_{+},\y)&=g_{1}(\x,\y),\\
(\x,\y_{+})&=g_{2}(\x,\y).
\label{def:g12:als}
\end{aligned}
\end{equation}
The mapping $g_{1},~g_{2}$ are well-defined in the whole domain $\R^{n}\times \R^{m}$,  in view of strong convexity and coercivity of the objective function in \eqref{eqn:new:BPAM}.
Then B-PAM \eqref{eqn:new:BPAM} can written as
\begin{align}
(\x_{k},\y_{k})=g(\x_{k-1},\y_{k-1}),\quad\tmop{where}~g:=g_{2}\circ g_{1}.
\label{def:g:bals}
\end{align}
By the chain rule, to compute the  the Jacobian matrix $Dg$, we need to compute $Dg_{1}$ and $Dg_{2}$, respectively.
We first compute $Dg_{1}$. First of all, by the definition of $g_1$, i.e., $(\x_{+},\y)=g_{1}(\x,\y)$, we have
\[
Dg_{1}(\x,\y)=\begin{bmatrix} {\partial \x_{+}}/{\partial \x^{\top}} & {\partial \x_{+}}/{\partial \y^{\top}} \\
\zero & \eye_{m}  \end{bmatrix}.
\]
Second, to compute ${\partial \x_{+}}/{\partial \x^{\top}}$ and ${\partial \x_{+}}/{\partial \y^{\top}}$, we  apply the implicit function theorem to the  zero-gradient condition \eqref{eqn:first:bals}
\begin{align*}
\nabla_{\x} h(\x_{+},\y)=\nabla_{\x} h(\x,\y)-\eta \nabla_{\x}f(\x_{+},\y),
\end{align*}
and obtain that
\begin{align*}
\left(\nabla^{2}_{\x\x}h(\x_{+},\y)+\eta\nabla^2_{\x\x}f(\x_{+},\y)\right)\frac{\partial \x_{+}}{\partial \x^{\top}}
&=\nabla^{2}_{\x\x}h(\x,\y),
\\
\left(\nabla^{2}_{\x\x}h(\x_{+},\y)+\eta\nabla^2_{\x\x}f(\x_{+},\y)\right)\frac{\partial \x_{+}}{\partial \y^{\top}}
&=\nabla^2_{\x\y}h(\x,\y)-\nabla^2_{\x\y}h(\x_{+},\y)  \!-\!\eta\nabla^2_{\x\y}f(\x_{+},\y)
\end{align*}
Then in view of the strong bi-convexity of  $h$, we can further simply the above equations as
\begin{align*}
\frac{\partial \x_{+}}{\partial \x^{\top}}
&\!=\!\!\left(\nabla^{2}_{\x\x}h(\x_{+}\!,\y)\!+\!\eta\nabla^2_{\x\x}f(\x_{+},\y)\right)^{\!-1}\!\nabla^{2}_{\x\x}h(\x,\y),
\\
\frac{\partial \x_{+}}{\partial \y^{\top}}
&\!=\!\!\left(\nabla^{2}_{\x\x}h(\x_{+}\!,\y)\!+\!\eta\nabla^2_{\x\x}f(\x_{+},\y)\right)^{\!-1}\!\!
(\nabla^2_{\x\y}h(\x,\y)\!-\!\nabla^2_{\x\y}h(\x_{+},\y)  \!-\!\eta\nabla^2_{\x\y}f(\x_{+},\y)).
\end{align*}
Similar arguments can be used to compute ${\partial \y_{+}}/{\partial \x^{\top}}$ and ${\partial \y_{+}}/{\partial \y^{\top}}$.  As a result, we have
\begin{equation}
\begin{aligned}
Dg_{1}(\x,\y)
=&\begin{bmatrix}
\left(\nabla^{2}_{\x\x}h(\x_{+},\y)+\eta\nabla^2_{\x\x}f(\x_{+},\y)\right)^{-1} & \\  &\ \eye_{m}
\end{bmatrix}
\\
&
\begin{bmatrix}
\nabla^{2}_{\x\x}h(\x,\!\y) &
\nabla^2_{\x\y}h(\x,\!\y)-\nabla^2_{\x\y}h(\x_{+},\y) -\eta\nabla^2_{\x\y}f(\x_{+},\y)
\\
\zero &\eye_{m}
\end{bmatrix},
\\
Dg_{2}(\x,\y)
=&
\begin{bmatrix}
\eye_{n} &
\\
&
\left(\nabla^{2}_{\y\y}h(\x,\y_{+})+\eta\nabla^2_{\y\y}f(\x,\y_{+})\right)^{-1}
\end{bmatrix}
\\
&
\begin{bmatrix}
\eye_{n} & \zero \\
\nabla^2_{\y\x}h(\x,\y)-\nabla^2_{\y\x}h(\x,\y_{+})  -\eta\nabla^2_{\y\x}f(\x,\y_+)
&
\nabla^{2}_{\y\y}h(\x,\y)
\end{bmatrix}.
\label{Dg:als}
\end{aligned}
\end{equation}
Therefore, using chain rule, we have
\begin{equation}
\begin{aligned}
Dg(\x,\y)=Dg_{2}(g_{1}(\x,\y))Dg_{1}(\x,\y).
\label{def:Dg:bals}
\end{aligned}
\end{equation}

Now by \Cref{thm:jason}, to show the mapping $g$ can almost surely avoid the strict saddles,  it suffices to show the following conditions:
\paragraph{\bf (1) Showing  $g$ is a continuously differentiable mapping}
This is because $D g$ in \eqref{def:Dg:bals} is continuous by the implicit function theorem.

\smallskip\paragraph{\bf (2) Showing $\det(D g)\neq 0$ in the whole domain}
Using the chain rule, we have $Dg=Dg_{2}Dg_{1}$, where  $Dg_{1}$ and $Dg_{2}$ are square matrices defined in  \eqref{Dg:als}.  First, $Dg_{1}$ is nonsingular in the whole domain in view of its upper-triangular block structure and the strong bi-convexity of $h$. Similarly, $Dg_2$ is also nonsingular in the entire domain.  This completes the proof.

\smallskip\paragraph{\bf (3) Showing any strict saddle is an unstable fixed point} First, since $(\x_+,\y_+)=(\x^\star,\y^\star),(\x,\y)=(\x^\star,\y^\star)$ satisfy the optimality conditions of both $g_1$ and $g_2$ (cf. \eqref{B-PAM:1} and \eqref{B-PAM:2}) and $g=g_2\circ g_1$ is well-defined by \Cref{lem:general:decrease:coordinate}, $(\x^\star,\y^\star)$ is a fixed point of $g$. 

It remains to show $\Sp(Dg(\x^{\star},\y^{\star}))>1$. From \eqref{def:Dg:bals}, we have
\begin{align*}
&Dg(\x^{\star},\y^{\star})
=Dg_{2}(\x^{\star},\y^{\star})Dg_{1}(\x^{\star},\y^{\star})\\
&=
\begin{bmatrix}
\eye_{n} &
\\
&
\left(\mH_2+\eta\mF_{22}\right)^{-1}
\end{bmatrix}
\begin{bmatrix}
\eye_{n} & \zero \\
-\eta\mF_{21}
&
\mH_2
\end{bmatrix}
\begin{bmatrix}
\left(\mH_1+\eta\mF_{11}\right)^{-1} &
\\
& \eye_{m}
\end{bmatrix}
\begin{bmatrix}
\mH_1
&
-\eta\mF_{12}
\\
\zero & \eye_{m}
\end{bmatrix}
\\
&=
\begin{bmatrix}
(\mH_1+\eta\mF_{11})^{-1} & \zero
\\
(\mH_1+\eta\mF_{11})^{-1} (-\eta\mF_{21})(\mH_2+\eta\mF_{22})^{-1}
& (\mH_2+\eta\mF_{22})^{-1}\mH_2
\end{bmatrix}
\begin{bmatrix}
\mH_1 & -\eta \mF_{12}\\
& \eye_{m}
\end{bmatrix}
\\
&=
\begin{bmatrix}
(\mH_1+\eta\mF_{11})^{\!-1} & \zero
\\
(\mH_1+\eta\mF_{11})^{\!-1} (\!-\eta\mF_{21})(\mH_2+\eta\mF_{22})^{\!-1}
& (\mH_2\!+\!\eta\mF_{22})^{-1}
\end{bmatrix}
\begin{bmatrix}
\eye_{n} &\\
& \mH_{2}
\end{bmatrix}
\begin{bmatrix}
\mH_1 & \!-\eta \mF_{12}\\
& \eye_{m}
\end{bmatrix}
\\
&=
\begin{bmatrix}
\mH_{1}+\eta\mF_{11} & \zero\\
\eta\mF_{21} & \mH_{2}+\eta\mF_{22}
\end{bmatrix}^{-1}
\begin{bmatrix}
\mH_{1} & -\eta\mF_{12}\\
& \mH_{2}
\end{bmatrix}
\end{align*}
with $\mF_{ij}$, $\mH_1,~\mH_2$ defined in \eqref{def:f}. As previously, it suffices to show that $\det(Dg(\x^{\star},\y^{\star})-\mu\eye)=0$ with  $|\mu|>1$. Towards this end, we observe that
\begin{align*}
&\det(Dg(\x^{\star},\y^{\star})-\mu\eye)=0\\
\iff &
\det\left(\begin{bmatrix}
\mH_{1}+\eta\mF_{11} & \zero\\
\eta\mF_{21} & \mH_{2}+\eta\mF_{22}
\end{bmatrix}^{-1}
\begin{bmatrix}
\mH_{1} & -\eta\mF_{12}\\
& \mH_{2}
\end{bmatrix}-\mu\eye\right)=0
\\
\iff&
\det\left(
\begin{bmatrix}
\mH_{1} & -\eta\mF_{12}\\
& \mH_{2}
\end{bmatrix}-\mu\begin{bmatrix}
\mH_{1}+\eta\mF_{11} & \zero\\
\eta\mF_{21} & \mH_{2}+\eta\mF_{22}
\end{bmatrix}\right)=0
\\
\iff&
\det\left(
\begin{bmatrix}
(\mu-1)\mH_{1}+\mu\eta\mF_{11} & \eta\mF_{12}\\
\mu\eta\mF_{21}& (\mu-1)\mH_{2}+\mu\eta\mF_{22}
\end{bmatrix}\right)=0
\\
\iff&
\det\left(
\begin{bmatrix}
\eye_{n} \!\!\!& \\
\!\!\!& \sqrt{\mu} \eye_{m}
\end{bmatrix}
\begin{bmatrix}
(\mu\!-\!1)\mH_{1}\!+\!\mu\eta\mF_{11} \!\!\!&\sqrt{\mu}\eta\mF_{12}\\
\sqrt{\mu}\eta\mF_{21} \!\!\!& (\mu\!-\!1)\mH_{2}\!+\!\mu\eta\mF_{22}
\end{bmatrix}
\begin{bmatrix}
\eye_{n} \!\!\!& \\
\!\!\!& \sqrt{\mu} \eye_{m}
\end{bmatrix}^{-1}\right)=0
\\
\iff&
\det\Bigg(
\underbrace{\begin{bmatrix}
(\mu-1)\mH_{1}+\mu\eta\mF_{11} &\sqrt{\mu}\eta\mF_{12}\\
\sqrt{\mu}\eta\mF_{21}& (\mu-1)\mH_{2}+\mu\eta\mF_{22}
\end{bmatrix}}_{\mJ(\mu)}\Bigg)=0,
\end{align*}
which holds if and only if  $\mJ(\mu)$ has a zero eigenvalue for some $|\mu|>1$.
For this purpose, we observe $\mJ(\mu)$ at two particular values of $\mu$: 
\begin{align*}
\mJ(1)&=\eta
\begin{bmatrix}
\mF_{11} &\mF_{12} \\
\mF_{21} &\mF_{22}
\end{bmatrix}=\eta\nabla^{2}f(\x^{\star},\y^{\star})
,\quad
\lim\limits_{\mu\to\infty}\frac{\mJ(\mu)}{\mu}=
\begin{bmatrix}
\mH_{1}+\eta\mF_{11} &\\
& \mH_{2}+\eta\mF_{22}
\end{bmatrix}.
\end{align*}
Since $(\x^{\star},\y^{\star})$ is a strict saddle of $f$, we have $\lambda_{\min}(\mJ(1))<0$, and since $\mH_{1}\pm\eta\mF_{11}$ and $\mH_{2}\pm\eta\mF_{22}$ are positive definite matrices (by assumptions), we also have $\lim_{\mu}\lambda_{\min}(\mJ(\mu))>0$. Further, note that $\mJ(\mu)$ is a
symmetric and continuous (w.r.t. $\mu$) matrix, which implies that the eigenvalues of $\mJ(\mu)$ are real-valued and continuous w.r.t. $\mu$ (cf. \cite[Theorem 5.1]{kato2013perturbation}). Then a continuity argument immediately complete the proof that $\mJ(\mu)$ has a zero eigenvalue for some $\mu>1$ and hence $|\mu|>1$.
\end{proof}

\section{Closed-form Implementations for Fourth-degree Polynomial Functions} \label{sec:implementations}

In this section, we provide closed-form solutions for efficiently implementing  B-GD and B-PALM. As the development of the closed-form solution for  B-PPM  and B-PAM rely on the specific form of the objective function,  we mainly focus on  B-GD  and B-PALM. Also, since most interesting problems in machine learning or signal processing (such as matrix PCA, matrix sensing and matrix completion, etc.) admit a fourth-degree (or (2,2)th-degree) polynomial objective function, let us focus on these two general cases of objective functions. We remark that it is not difficult to consider an arbitrary polynomial objective function. The only issue is that there might be no closed-form solution since the zero-gradient equation is a high-degree polynomial equation, but one can nevertheless solve the optimality condition using line-search algorithms.

\subsection{Closed-form Implementations for B-GD} By \Cref{lem:h:exist}, set $h(\x)=\frac{1}{4}\|\x\|_2^4+\frac{1}{2}\|\x\|_2^2+1$ to achieve the second-order convergence for a fourth-degree polynomial function.
In this case, B-GD is a smart version of ``GD" equipped with the ability of line search algorithm that can adaptively choose the step size according to the norm of the current iterate. Note that such a line-search strategy is much more efficient than the traditional line search method, as it can automatically choose the step size with a closed-form solution. 
\begin{theorem}\label{pro:bgd}
Suppose $f(\x)$ is any fourth-degree polynomial and set $h(\x)=\frac{1}{4}\|\x\|_2^4+\frac{1}{2}\|\x\|_2^2+1$. Then B-GD has the following closed-form implementation:
\begin{align}
\x_{k}=
\Pi\left(\left(\|\x_{k-1}\|_2^2 + 1\right)\x_{k-1} -\eta \nabla f(\x_{k-1})\right),
\label{eqn:bgd:update}
\end{align}
where  
\begin{align}
\Pi(\x):= \x   \frac{\tau(\|\x\|_2)}{\|\x\|_2}  
\label{def:Pi}
\end{align}
 with $\tau(\cdot)$  defined in \eqref{def:tau}.
\end{theorem}
To prove this result, we require the following lemma to solve a third-degree polynomial equation.
\begin{lem}
For any $a\ge0$, the cubic polynomial
$
 t^3 + t  = a
$
has a unique  solution
\begin{align}
\tau(a):= 
\frac{\sqrt[3]{2} \left(\sqrt{81 a^2+12}+9 a\right)^{2/3}-2 \sqrt[3]{3}}{6^{2/3} \sqrt[3]{\sqrt{81 a^2+12}+9 a}}
\label{def:tau}.
\end{align}
\label{lem:cubic}
\end{lem}
\begin{proof}[Proof of \Cref{lem:cubic}]
First it can be verified that $\tau(a)$ is a solution by direct computations and the uniqueness follows from  the strictly increasing property of the function $t^3+t$.
\end{proof} 
\begin{proof}[Proof of \Cref{pro:bgd}] 
First plug $\nabla h(\x) = (\|\x\|_2^2 + 1)\x$  in the zero-gradient condition of \eqref{eqn:bgd1}:  $ \eta \nabla  f(\x_{k-1}) - \nabla h(\x_{k-1}) + \nabla h(\x_{k}) = \vzero$, which will be reduced to a equation of the type $t^3+t=a$. Then the proof follows from \Cref{lem:cubic}.
\end{proof}

\subsection{Closed-form Implementations for B-PALM}
Let us now consider a more difficult case for a bi-variable polynomial function of $(4,4)$ degree. By \Cref{lem:h:exist}, it is sufficient to set $h(\x,\y)=(\frac{1}{4}\|\x\|_2^4+\frac{1}{2}\|\x\|_2^2+1)(\frac{1}{4}\|\y\|_2^4+\frac{1}{2}\|\y\|_2^2+1)$ to achieve the second-order convergence. Similarly, the B-PALM reduces to an alternating ``GD" with equipped the line search ability for each subproblem. 
\begin{theorem}
Suppose $f(\x,\y)$ is any $(4,4)$-degree polynomial and set
$
h(\x,\y)=(\frac{1}{4}\|\x\|_2^4+\frac{1}{2}\|\x\|_2^2+1)(\frac{1}{4}\|\y\|_2^4+\frac{1}{2}\|\y\|_2^2+1).
$
Then B-PALM has closed-form implementations:
\begin{equation}
\label{closed:form:fxy:4}
\begin{aligned}
\x_{k}=&\Pi\!\left(\!\left(\|\x_{k-1}\|_2^2 \!+\! 1\right)\x_{k-1} \!-\!\frac{\eta}{\left(\|\y_{k-1}\|_2^4/4\!+\!\|\y_{k-1}\|_2^2/2\!+\!1\!\right)} \nabla_{\x} f(\x_{k-1},\y_{k-1})\right), 
\\
\y_{k}=&\Pi\!\left(\!(\|\y_{k-1}\|_2^2 \!+\! 1)\y_{k-1} \!-\!\frac{\eta}{\left(\|\x_{k}\|_2^4/4\!+\!\|\x_{k}\|_2^2/2\!+\!1\right)}\nabla_{\y} f(\x_{k},\y_{k-1})\!\right) 
\end{aligned}    
\end{equation}
 with $\Pi(\cdot)$  defined in \eqref{def:Pi}. 
\label{cor:BMF:fxy4}
\end{theorem}
\begin{proof}
The proof of  \Cref{cor:BMF:fxy4} is similar to that of \Cref{pro:bgd}.
\end{proof}

When the objective function $f(\x,\y)$ is a $(2,2)$th-degree polynomial, we can simplify the above closed-form updating formula. Remarkably,  most interesting problems in machine learning or signal processing admit a $(2,2)$th-degree polynomial objective function, e.g.,  matrix PCA, matrix sensing, matrix  completion, etc.
\begin{theorem}
Suppose $f(\x,\y)$ is any $(2,2)$th-degree polynomial and  set
$
h(\x,\y)=\left(\frac{1}{2}\|\x\|_2^2+1\right)\left(\frac{1}{2}\|\y\|_2^2+1\right).
$
Then  B-PALM has closed-form implementations:
\begin{equation}
\label{closed:form:fxy:2}
\begin{aligned}
\x_{k}=&\x_{k-1}-\frac{\eta}{\|\y_{k-1}\|_2^2/2+1}\nabla_{\x}f(\x,^{k-1},\y_{k-1}),
\\
\y_{k}=&\y_{k-1}-\frac{\eta}{\|\x_{k}\|_2^2/2+1}\nabla_{\y}f(\x_{k},\y_{k-1})
\end{aligned}    
\end{equation}
 with $\Pi(\cdot)$  defined in \eqref{def:Pi}.
\label{pro:bgd4}
\end{theorem}

\begin{proof}[Proof of \Cref{pro:bgd4}]
The proof directly follows from the zero-gradient condition of \eqref{update:gd:coordinate} in B-PALM:
\begin{align*}
\eta\nabla_{\x}f(\x_{k-1},\y_{k-1})+ \nabla_{\x} h(\x_{k},\y_{k-1})- \nabla_{\x} h(\x_{k-1},\y_{k-1})    &=0,
\\
\eta\nabla_{\y}f(\x_{k},\y_{k})+ \nabla_{\y} h(\x_{k},\y_{k})- \nabla_{\x} h(\x_{k},\y_{k-1})    &=0
\end{align*}    
with $\nabla_{\vx}h(\vx,\vy)=\left(\frac{1}{2}\|\y\|_2^2+1\right)\x$ and
$\nabla_{\vy}h(\vx,\vy)=\left(\frac{1}{2}\|\x\|_2^2+1\right)\y$.
\end{proof}

\section{Numerical Experiments} In this section, we test the proposed Bregman-divergence based algorithms by comparing with the traditional Euclidean-distance based methods (GD and PALM) on the following nuclear-norm regularized and rank-constrained optimization problem for both symmetric case and nonsymmetric case:
\begin{align}
\mX^\star_r=\argmin_{\mX\in\mathbb{S}_+^n~\tmop{or}~\mX\in\R^{n\times m}}\frac{1}{2}\|\mX-\mA\|_F^2+\lambda\|\mX\|_*~~\st~\rank(\mX)\le r
\label{eqn:pca}
\end{align}
The global optimal solution $\mX_r^\star$ can be obtained by a combination of a soft thresholding and a best rank-$r$ approximation (and an additional projection to the symmetric positive semidefinite cone $\mathbb{S}_{+}^n$ for the symmetric case):
\begin{align}
\mX_r^\star= 
\begin{cases}
\tmop{SVD}_r\left(\tmop{SoftTh}_\lambda(\tmop{Proj}_{\mathbb{S}_+^n}(\mA))\right)     & \text{for the symmetric case}
\\
 \tmop{SVD}_r\left(\tmop{SoftTh}_\lambda(\mA)\right) & \text{for the nonsymmetric case}
\end{cases}
\label{solution:svd}
\end{align}
where $\tmop{SVD}_r(\cdot)$ denotes the best rank-$r$ approximation and $\tmop{SoftTh}_\lambda(\cdot)$ is defined as the soft thresholding by first decreasing the singular values by $\lambda$ and then removing the negative ``negative" ones. Inspired by \Cref{sec:BMF}, another way to deal with the low-rank constraint is applying BMF optimization method to the original rank-constrained  problem \eqref{eqn:pca} and solving
\begin{equation}
\begin{aligned}
\minimize_{\mU\in\R^{n\times r}}f(\mU)&:=\frac{1}{2}\|\mU\mU^\top-\mA\|_F^2+\lambda \|\mU\|_F^2,\\
\minimize_{\mU\in\R^{n\times r},\mV\in\R^{m\times r}}f(\mU,\mV)&:=\frac{1}{2}\|\mU\mV^\top-\mA\|_F^2+\frac{\lambda}{2}(\|\mU\|_F^2+\|\mV\|_F^2).
\label{eqn:factorization}
\end{aligned}
\end{equation}
We remark that the proposed Bregman-divergence methods are guaranteed to globally minimize the BMF problem \eqref{eqn:factorization}. This is because it has been established in \cite{li2018non} that every second-order stationary point of \eqref{eqn:factorization} is globally optimal. Therefore, the convergence to a second-order stationary point of the proposed Bregman methods established in \Cref{cor:BMF} implies the global optimal convergence.

To implement the Bregman-divergence based algorithms, we can use the closed-form updating formulas  \Cref{pro:bgd} (for B-GD) and \Cref{pro:bgd4} (for B-PALM) to minimize $f(\mU)$ and $f(\mU,\mV)$, respectively. To be fair, we will tune step sizes of all algorithms until achieve best perforamnce. In the experiments,  we set $n=m=1000$, $r=2$, $\lambda=1$,  generate a symmetric matrix $\mA$ pointwisely i.i.d.  from  $\calN(0,1)$, and run GD, B-GD, PALM and B-PALM on the BMF problem \eqref{eqn:factorization}. To test the main advantage (i.e., robustness of initialization and step size) of the Bregman-divergence based algorithms, we will perform the two sets of experiments: one with small initialization (pointwisely from i.i.d. $\calN(0,0.1)$) and one with large initialization (pointwisely from i.i.d. $\calN(0,10)$).

From \Cref{fig:bgd:PCA}, we can see that when the initialization is small, both types of algorithms perform pretty well. However, when the initialization is large, the traditional Euclidean-distanced based methods degrade drastically in term of the convergence speed and even sometimes fail to converge to a second-order stationary point, while the Bregman-divergence based methods can efficiently converge to the global optimal solutions in both cases of initializations. 
This is because the large initialization can give rise to a very large local Lipschitz constant, which then forces GD to use a very small step size, resulting in an extremely poor algorithm efficiency. Indeed, this is one important advantage of B-GD \cite{bolte2018first}  to allow adaptive step sizes. As a contrast, the Bregman-divergence based algorithms are equipped with the strength of the line search that can adaptively choose the step size according to the norm of the current iterate. See \Cref{pro:bgd} and \Cref{pro:bgd4}.

\begin{figure}[!ht]
\begin{tabular}{cc}
\!\!\includegraphics[width=0.47\textwidth]{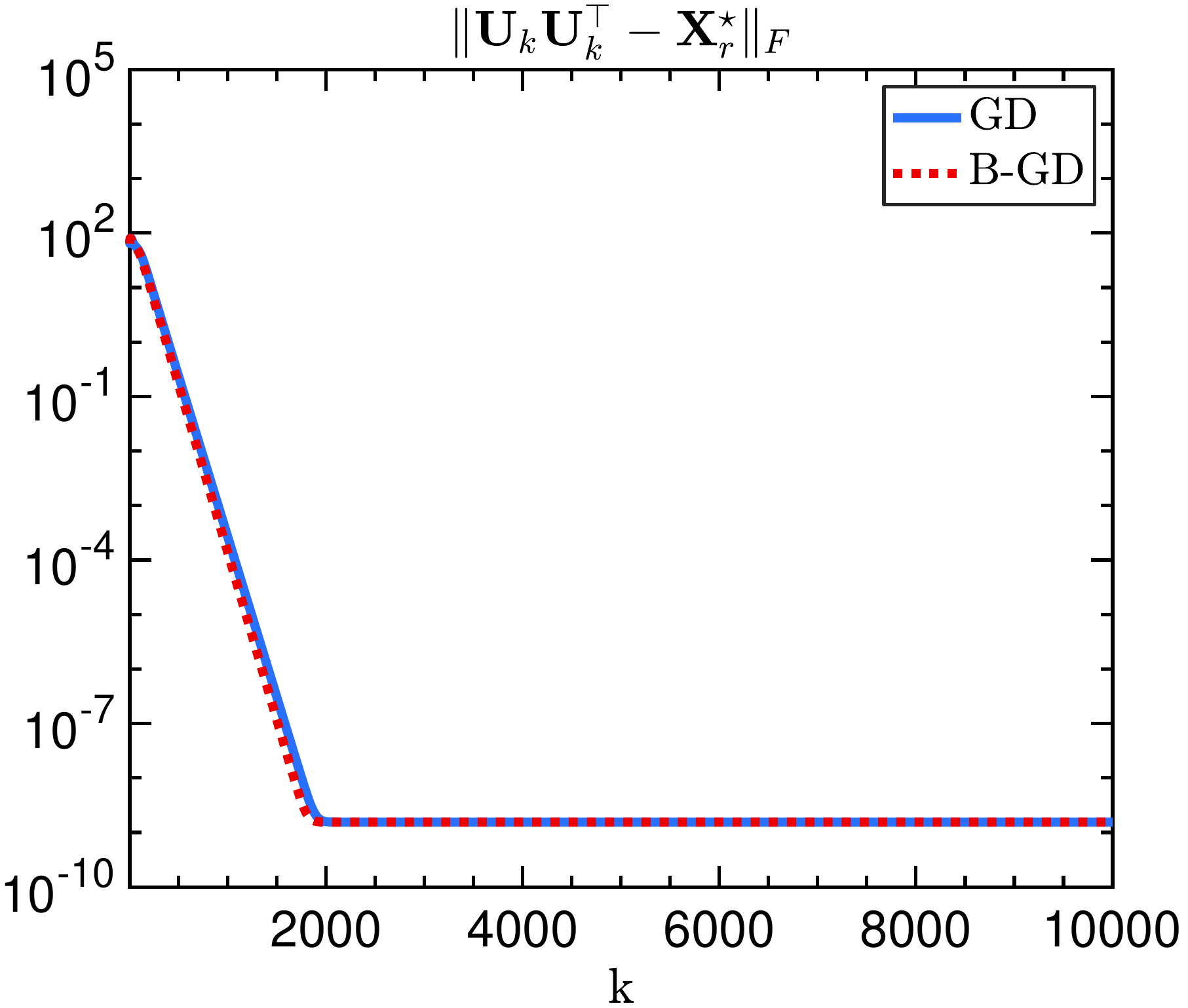}&
\includegraphics[width=0.47\textwidth]{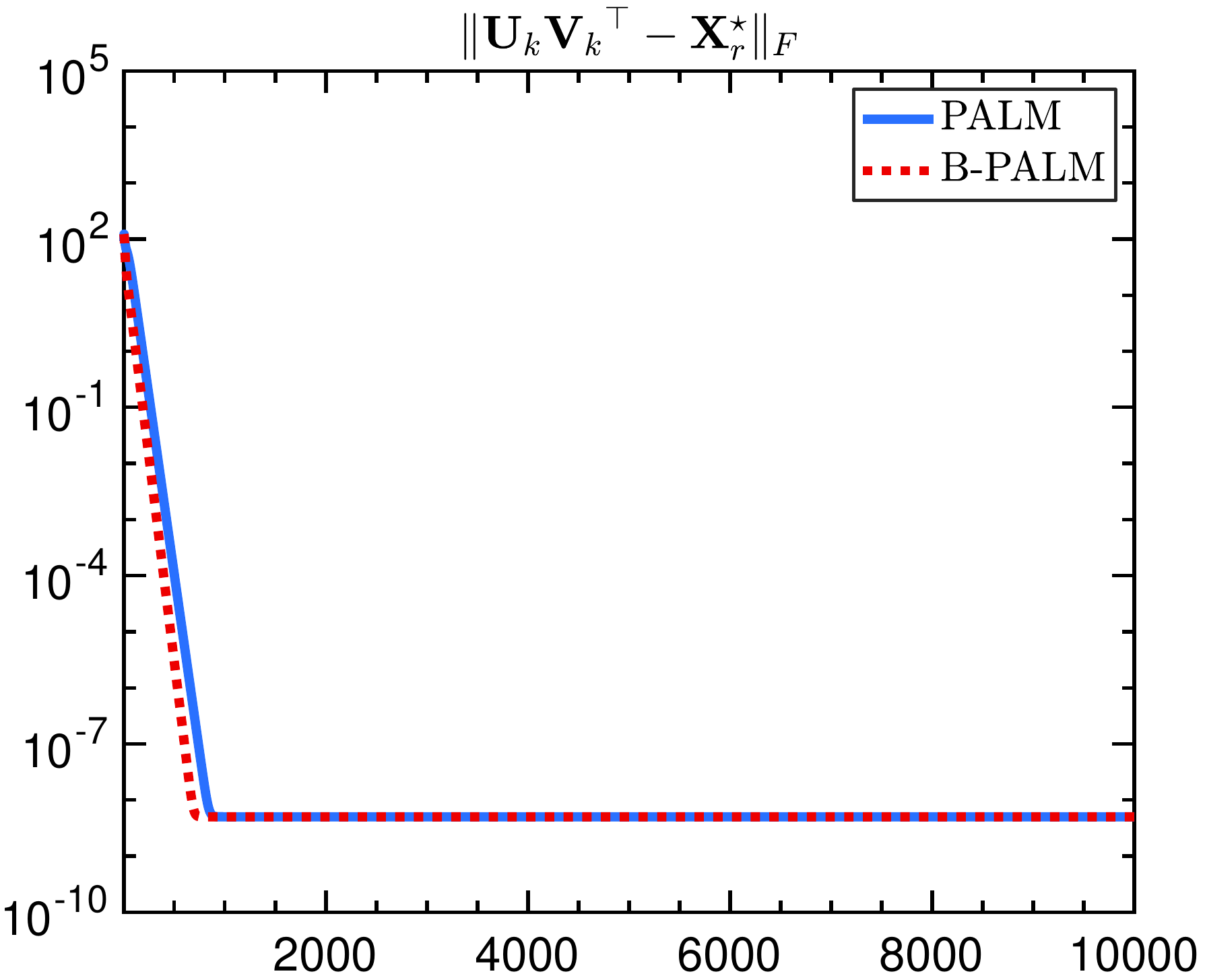}
\\
\!\!\includegraphics[width=0.47\textwidth]{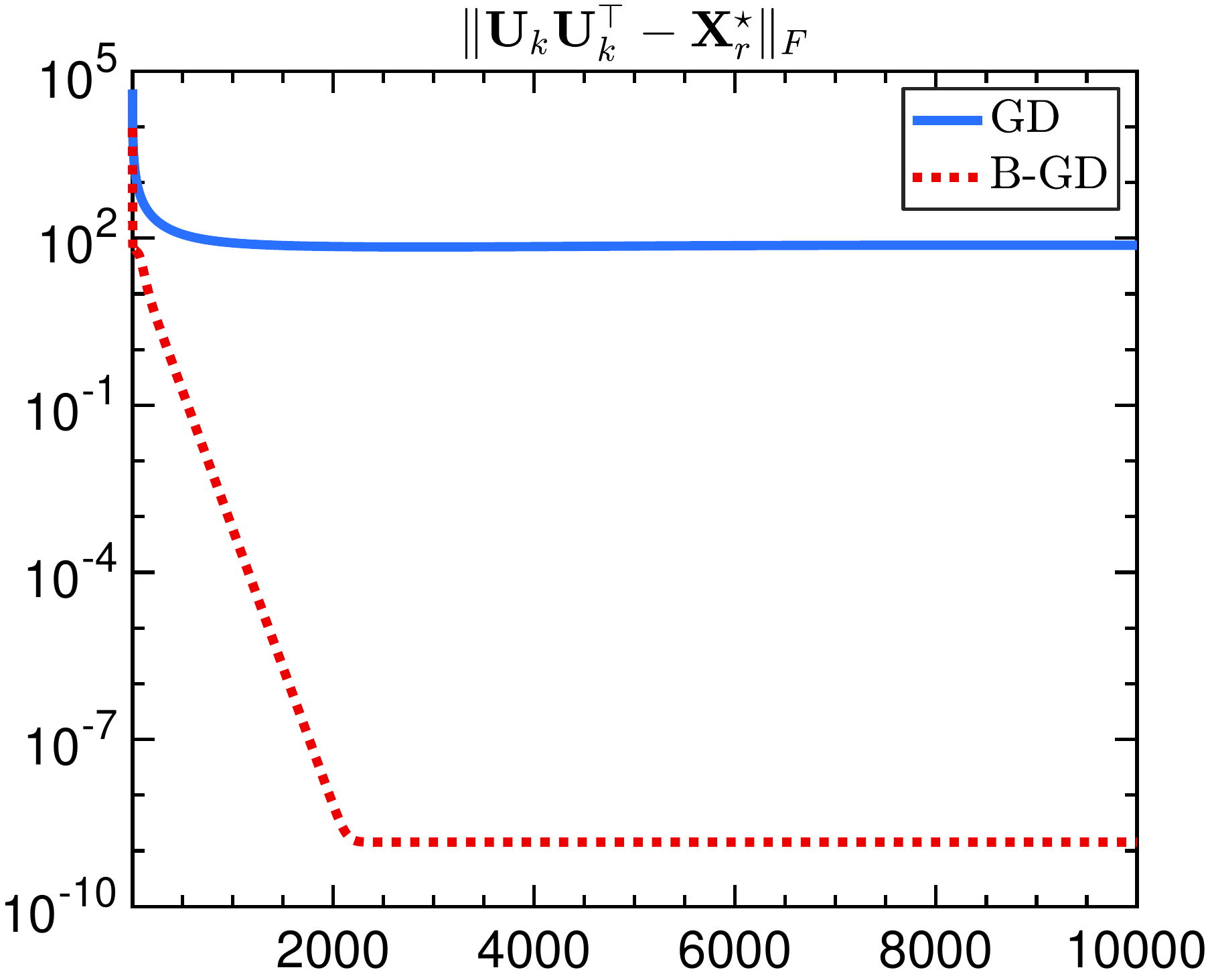}&
\includegraphics[width=0.47\textwidth]{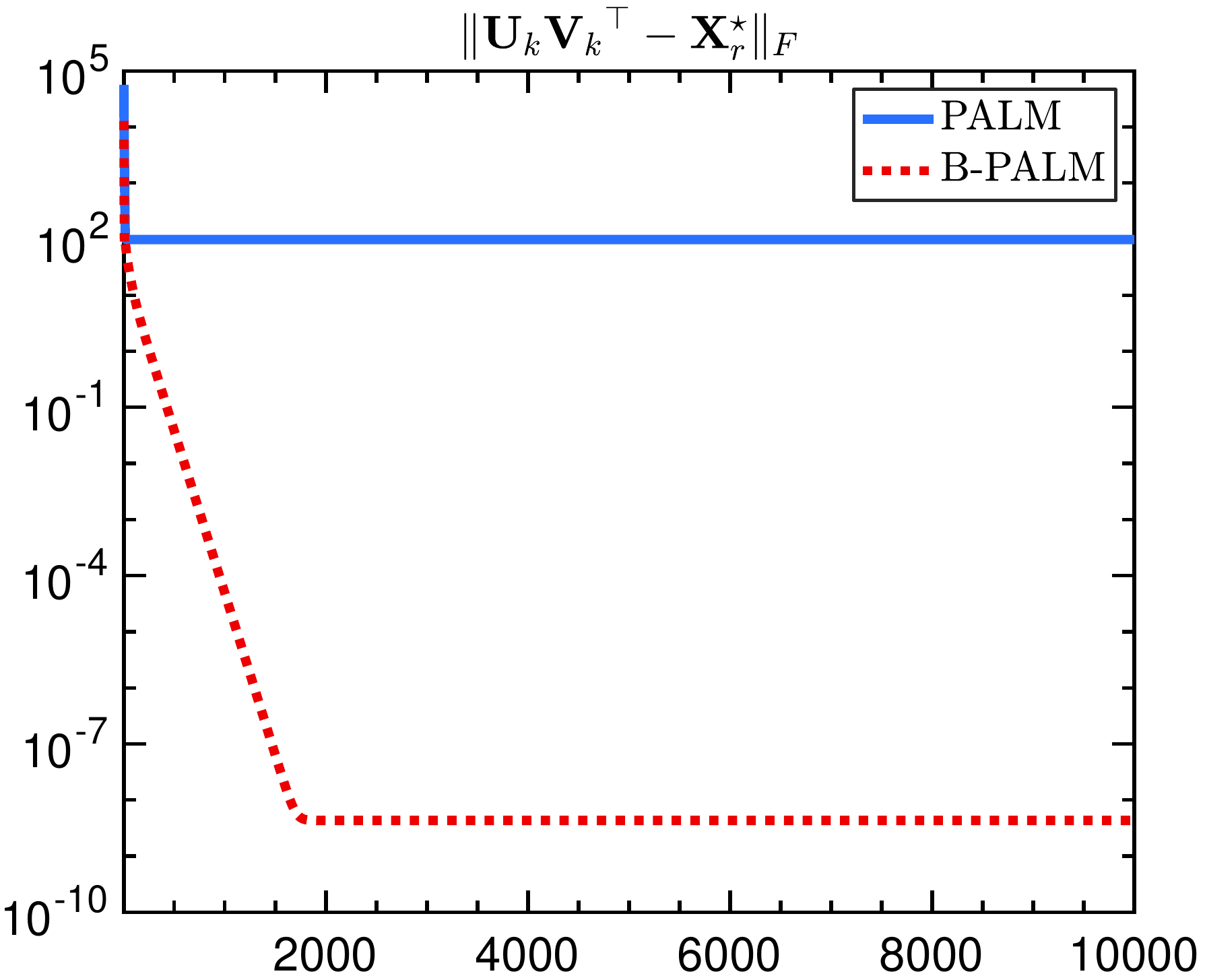}
%\\
%(a) & (b)
%\\
%\includegraphics[width=0.45\textwidth]{bregman_nonsymPCA_alsg_1_0402}&
%\includegraphics[width=0.45\textwidth]{bregman_nonsymPCA_alsg_2_0402}
%\\
%(c) & (d)
\end{tabular}
\caption{
{\bf Above Left}: GD vs B-GD with small initialization;
{\bf Above Right}: PALM vs B-PALM with small initialization. 
{\bf Below Left}: GD vs B-GD with large initialization;
{\bf Below Right}: PALM vs B-PALM with large initialization. }
\label{fig:bgd:PCA}
\end{figure}

%\begin{figure}[!ht]
%\begin{center}
%\begin{tabular}{cc}
%\includegraphics[width=0.45\textwidth]{bregman_symPCA_1_0402}&
%\includegraphics[width=0.45\textwidth]{bregman_symPCA_2_0402}
%\\
%(a) & (b)
%\\
%\includegraphics[width=0.45\textwidth]{bregman_nonsymPCA_alsg_1_0402}&
%\includegraphics[width=0.45\textwidth]{bregman_nonsymPCA_alsg_2_0402}
%\\
%(c) & (d)
%\end{tabular}
%\caption[haha]{
%({a}): $\mU_{0}(i,j)\stackrel{iid}{\sim}\calN(0,0.1)$;
%({b}): $\mU_{0}(i,j)\stackrel{iid}{\sim}\calN(0,10)$.
%({c}):  $\mU_{0}(i,j), \mV_{0}(i,j)\stackrel{iid}{\sim}\calN(0,0.1)$;
%({d}):  $\mU_{0}(i,j), \mV_{0}(i,j)\stackrel{iid}{\sim}\calN(0,10)$. }
%\label{fig:bgd:PCA}
%\end{center}
%\end{figure}

\appendix

\section{Proof of \Cref{lem:general:decrease}}\label{pf:lem:general:decrease}

\begin{proof}
We first show \eqref{eqn:general:decrease1}. For simplifying notations, denote $\x_{+}:=\x_{k}$ and $\x_{k-1}:=\x_{-}$. For the well-definedness, it suffices to show the solution of \eqref{eqn:general:decrease1} exists and is unique. First, since the objective function is continuous (as $f,h\in\mathcal{C}^{2}$), its level set
$
\tmop{Lev}_{\phi}(a):=\{\x: \phi(\x)\le a\}
$
is closed for any $a\in\R$,
where $\phi(\x):=f(\x_{-})+\lg \nabla f(\x_{-}), \x-\x_{-}\rg+D_{h}(\x,\x_{-})/\eta$. Second, when $h$ is super-coercive, we will show the objective function $\phi(\x)$ is coercive, which would imply the boundedness of the level set $\tmop{Lev}_{\phi}(a)$. Then together with the closedness of the level set, we can view \eqref{eqn:general:decrease1} as a minimization of a continuous function over a compact level set and hence the solution must exist. The uniqueness follows from the strong convexity of $\phi$ because $\nabla^{2}\phi=\nabla^{2}h$ and $h$ is strongly convex. Now, we show $\phi(\x)$ is coercive.

\begin{align*}
\phi(\x)
=&f(\x_{-})+\lg \nabla f(\x_{-}), \x-\x_{-}\rg+D_{h}(\x,\x_{-})/\eta
\\
=&f(\x_{-})+\lg \nabla f(\x_{-}), \x-\x_{-}\rg+(h(\x)-h(\x_{-})-\lg\nabla h(\x_{-}),\x-\x_{-}\rg)/\eta
\\
%=& \frac{h(\x)}{\eta} \!+\!\lg   \nabla f(\x_{-})\!-\!\frac{\nabla h(\x_{-})}{\eta}, \x\rg \!+\! (f(\x_{-})\!-\!\lg\nabla f(\x_{-}),\x_{-}\rg \!-\! \frac{h(\x_{-})}{\eta} \!+\!\lg\frac{\nabla h(\x_{-})}{\eta},\x_{-}\rg)
%\\
:=& {h(\x)}/{\eta}+\left\lg \a, \x\right\rg +b
\\
=& \|\x\|_{2} \left( {h(\x)}/\|\x\|_{2}/\eta + \left\langle\a, {\x}/{\|\x\|_{2}}\right\rangle\right)+b
\\
\ge&  \|\x\|_{2} \left( {h(\x)}/\|\x\|_{2}/\eta - \|\a \|_{2} \right)+b
\end{align*}
Now using the super-coercivity of $h$, we have  ${h(\x)}/{\|\x\|_{2}}/\eta > \|\a \|_{2}$ for any $\eta>0$ when $\|\x\|_{2}$ is large enough, implying that $\phi$ is coercive.

Now show the sufficient decrease property of \eqref{eqn:general:decrease1}.  By definition of of $\x_+$, we have
\begin{equation}
\begin{aligned}    
f(\x_{-})&= f(\x_{-})+\lg\nabla f(\x_{-}),\x-\x_{-} \rg\!+\!\frac{1}{\eta} D_{h}(\x,\x_{-})\Big|_{\x=\x_{-}}
\\
&\ge f(\x_{-})+\lg\nabla f(\x_{-}),\x_{+}-\x_{-} \rg\!+\!\frac{1}{\eta} D_{h}(\x_{+},\x_{-})
\\
&\ge f(\x_{+})\!-\!L_{f}D_{h}(\x_{+},\x_{-})\!+\!\frac{1}{\eta} D_{h}(\x_{+},\x_{-})
\\
&\ge f(\x_{+})\!+\!\left( \frac{1}{\eta}\!-\!L_{f} \right)\frac{\sigma}{2}\|\x_{+}\!-\!\x_{-}\|_{2}^{2}
\label{showing:coercive:using:super-coercive}
\end{aligned}
\end{equation}
where the second inequality is by the general descent lemma \eqref{eqn:general:descent:lemma} with $\y=\x_{-},\x=\x_{+}$ and the last inequality follows from the $\sigma$-strong convexity of $h$.

We now show \eqref{eqn:general:decrease2}.  Its well-definedness follows in the same way by showing that the objective function of \eqref{eqn:general:decrease2} is coercive (by using the same analysis as \eqref{showing:coercive:using:super-coercive} combined with the lower-boundedness of $f$) and strongly convex (since $f$ satisfies $L_{f}$-relative smoothness condition w.r.t. $h$ and $\eta\in(0,1/L_{f})$).  The sufficient decrease property follows  by
\begin{align*}
f(\x_{-})
= f(\x_{-})\!+\!\frac{1}{\eta} D_{h}(\x_{-},\x_{-})
&\ge f(\x_{+})\!+\!\frac{1}{\eta} D_{h}(\x_{+},\x_{-})
\ge f(\x_{+})\!+\!\frac{\sigma}{2\eta} \|\x_{+}-\x_{-}\|_{2}^{2}
\end{align*}
\end{proof}

\section{Proof of \Cref{lem:h:exist}}\label{sec:pf:lem:h:exist}

\subsection{Proof for the Single-variable Case}
First, any $d$th-degree polynomial function $f(\x)$ can be represented as
$
f(\x)=\sum_{k=0}^{d} \lg \calA_{k}, \x^{\otimes k}\rg,
$
where  $\otimes$ is the tensor/outer product, $\x^{\otimes k}:=\x\otimes\x\otimes\cdots\otimes\x$, and the coefficients of $k$th-degree monomials are arranged as a $k$th-order tensor $\calA_{k}\in \R^{n\times n \cdots\times n}$. For convenience, we denote $\x^{\otimes 0}=1$ and $\calA_0\in\R$. Further, by the supersymmetry, $\x^{\otimes k}$,  we can assume $\calA_{k}$ for $k\ge 2$ are also supersymmetric tensors since otherwise we can rearrange them as supersymmetric tensors.

Now, we show that polynomial $f$ with the particular $h(\x)=\frac{\alpha}{{d}}\|\x\|_{2}^{d}+\frac{\beta}{2}\|\x\|_{2}^{2}+1$ in \eqref{eqn:h:exist}
satisfies all  Assumptions \ref{assump:f}--\ref{assump:h}.

\smallskip\paragraph{\bf (1) Showing \Cref{assump:f}}
This directly follows from that $f$ is a lower-bounded polynomial function and any polynomial function is a twice differentiable KL function. 

\smallskip\paragraph{\bf (2) Showing \Cref{assump:h}}

First of all, we show $h$ is a $\calC^2$, super-coercive, and strongly convex function.
Let us compute the Hessian of $h(\x)$:
\begin{align*}
\nabla^{2}h(\x)&=\frac{\alpha}{{d}} {d}({d}-2)\|\x\|_{2}^{{d}-4}\x\x^{\top} +\left(\frac{\alpha}{{d}} {d}\|\x\|_{2}^{{d}-2}+2\frac{\sigma}{2}\right)\eye_{n}
\\
&={\alpha}({d}-2)\|\x\|_{2}^{{d}-2}\frac{\x}{\|\x\|_2}\frac{\x^\top}{\|\x\|_2}+\left({\alpha}\|\x\|_{2}^{{d}-2}+{\sigma}\right)\eye_{n}
\end{align*}
which immediately implies that $h(\x)$ is $\sigma$-strong convex and
that $h\in\calC^2$ for any $d\ge 2$. Now we show the super-coercivity:
\[\lim_{\|\x\|_{2}\to\infty} \frac{h(\x)}{\|\x\|_{2}}= \lim_{\|\x\|_2\to\infty}\left(\frac{\alpha}{{d}}\|\x\|_{2}^{{d}-1}+\frac{\sigma}{2}\|\x\|_{2}+\frac{1}{\|\x\|_2}\right)\ge \lim_{\|\x\|_2\to\infty} \frac{\sigma}{2}\|\x\|_{2} =\infty\] 

Second, we show the relative smoothness of $f$ w.r.t. $h$.
It suffices to show that there exists a constant $L_{f}$ such that
$
L_{f}\nabla^{2}h(\x)\pm \nabla^{2}f(\x)\succeq 0,~\forall~ \x.
$
Towards that end, we first compute its Hessian matrix of  $f(\x)=\sum_{k=0}^{d} \lg \calA_{k}, \x^{\otimes k}\rg
$ as 
\begin{align}
\nabla^{2} f(\x)=\sum_{k=2}^{d} k(k-1) \calA_{k}\times_{1}\x \times_{2}\x\times_{3}\x\cdots\times_{k-2}\x
\label{eqn:hess:f}
\end{align}
where we have used that $\calA_k$ is a supersymmetric tensor. Here ${\times}_k$ denotes the $k$th-mode tensor-vector product for any $N$th-order tensor $\calA$ (cf. \cite{kolda2009tensor}).
%, i.e., 
%\[\calA\times_k\x=\left[\sum_{j=1}^{n}\calA(i_1,\cdots,i_{k-1},j,i_{k+1},\cdots,i_N)\x(j)\right]_{i_1,\cdots,i_{k-1},i_{k+1},\cdots,i_N}\]

Now, we can use the triangle inequality and the definition of tensor spectral norm to control its Hessian spectral norm:
\begin{equation}
\begin{aligned}
\|\nabla^{2}f(\x)\|
&\le \sum_{k=2}^{d} k(k-1) \|\calA_{k}\times_{1}\x \times_{2}\x\times_{3}\x\cdots\times_{k-2}\x\|
\\
&\le \sum_{k=2}^{d} k(k-1)  \|\calA_{k}\| \|\x\|_{2}^{k-2}
\\
&\le\sum_{k=2}^d k(k-1)\|\calA_k\| (1+\|\x\|_2^{d-2})
\end{aligned}
\label{eqn:bounded:hess:polynomial}
\end{equation}
Meanwhile, by the previous expression of $\nabla^2 h(\x)$, we  have 
$
\nabla^{2}h(\x)\succeq (\alpha \|\x\|_{2}^{{d}-2}+\sigma)\eye_{n}.
$
Therefore,  $f$ satisfies $L_f$-relative smoothness condition w.r.t. $h$ for any
$
L_{f}\ge
\sum_{k=2}^{d}  k(k-1)  \|\calA_{k}\| \max\left\{ \frac{1}{\sigma},\frac{1}{\alpha}  \right\}.
$

This completes the proof of showing  Assumptions \ref{assump:f}--\ref{assump:h}.

\subsection{Proof for the Bi-variable Case}
First, any $(d_1,d_2)$th-degree polynomial function $f(\x,\y)$ can be represented as
$
f(\x,\y)=\sum_{i=0}^{d_1} \sum_{j=0}^{d_2} \lg \calA_{i,j}, \x^{\otimes i}\otimes\y^{\otimes j}\rg,
$
where the coefficients of $(i,j)$th-degree monomials are arranged as $\calA_{i,j}\in \prod_{k=1}^i \R^{n} \times \prod_{k=1}^j\R^{m}$.
For convenience, we denote $\x^{\otimes 0}=\y^{\otimes 0}=1$ and $\calA_{0,0}\in\R$. 
Further, due to supersymmetric tensors $\x^{\otimes i}$ and $\y^{\otimes j}$,  we can always assume $\calA_{i,j}$ for $i\ge 2$ or $j\ge 2$ as bi-supersymmetric tensors, i.e., those entries  $\calA_{i,j}(k_1,\cdots,k_i,k_{i+1},\cdots,k_{i+j})$ have the same value despite the order of $(k_1,k_2,\cdots,k_j)$ and the order of $(k_{i+1},k_{i+1},\cdots,k_{i+j})$.

\smallskip\paragraph{\bf (1) Showing $h$ is bi-super-coercive and $\sigma$-strongly bi-convex} Observe that
\begin{align*}
&\lim_{\|\x\|_2\to\infty} \frac{h(\x,\y)}{\|\x\|_2}\ge \lim_{\|\x\|_2\to\infty}\frac{\sigma}{2}\|\x\|_2=\infty,\quad\lim_{\|\y\|_2\to\infty} \frac{h(\x,\y)}{\|\y\|_2}\ge \lim_{\|\x\|_2\to\infty}\frac{\sigma}{2}\|\y\|_2=\infty
\end{align*}
which implies that $h(\x,\y)$ is bi-super-coercive. It remains to show that $h(\x,\y)$ is $\sigma$-strongly bi-convex. Towards that end, we compute the partial Hessians of $h(\x,\y)$:
\begin{align}
\nabla^2_{\x\x}h(\x,\y)&\!=\!\left(\!\frac{\alpha}{d_2}\|\y\|_{2}^{d_2}\!+\!\frac{\sigma}{2}\|\y\|_{2}^{2}
\!+\!1\!\right)\!\!\left(\!{\alpha}({d_1}\!-\!2)\|\x\|_{2}^{{d_1}\!-\!2}\frac{\x}{\|\x\|_2}\frac{\x^\top}{\|\x\|_2}\!+\!(\!{\alpha}\|\x\|_{2}^{{d_1}\!-\!2}\!\!+\!{\sigma})\eye_{n}\!\!\right)
\label{partial:hxx}
\\
\nabla^2_{\y\y}h(\x,\y)&\!=\!\left(\!\frac{\alpha}{d_1}\|\x\|_{2}^{d_1}\!+\!\frac{\sigma}{2}\|\x\|_{2}^{2}
\!+\!1\!\right)\!\!\left(\!{\alpha}({d_2}\!-\!2)\|\y\|_{2}^{{d_2}\!-\!2}\frac{\y}{\|\y\|_2}\frac{\y^\top}{\|\y\|_2}\!+\!({\alpha}\|\y\|_{2}^{{d_2}-2}\!\!+\!{\sigma})\eye_{m}\!\!\right)
\label{partial:hyy}
\end{align}
This then implies that 
$
\nabla^2_{\x\x}h(\x,\y)\succeq  \sigma \eye_n$ 
and  
$
\nabla^2_{\y\y}h(\x,\y)\succeq  \sigma \eye_m.
$
Therefore, $h(\x,\y)$ is $\sigma$-strongly bi-convex. This completes the proof of Part 1.

\smallskip\paragraph{\bf (2) Showing $f$ is relative bi-smooth w.r.t. $h$} In one way, using the bi-supersymmetry of $\calA_{i,j}$, we have
\begin{equation*}
\begin{aligned}
\|\nabla_{\x\x}^2f(\x,\y)\|
&\le \sum_{i=2}^{d_1}\sum_{j=0}^{d_2} i(i-1) \|\calA_{i,j}\times_{1}\x \times_{2}\x\cdots\times_{i-2}\x\times_{i+1}\y\cdots\times_{i+j}\y\|
\\
&\le \sum_{i=2}^{d_1}\sum_{j=0}^{d_2} i(i\!-\!1) \|\calA_{i,j}\| \|\x\|_2^{i-2}\|\y\|_2^{j}
\\
&\le      (1\!+\!\|\x\|_2^{d_1-2})(1\!+\!\|\y\|_2^{d_2})
\sum_{i=2}^{d_1}\sum_{j=0}^{d_2} i(i\!-\!1) \|\calA_{i,j}\| 
\end{aligned}
\end{equation*}
Similarly, $\|\nabla_{\y\y}^2f(\x,\y)\|
\le (1+\|\x\|_2^{d_1}) (1+\|\y\|_2^{d_2-2}) \sum_{i=0}^{d_1}\sum_{j=2}^{d_2} j(j-1) \|\calA_{i,j}\|.
$

In another way, by \eqref{partial:hxx} and \eqref{partial:hyy}, 
\begin{equation}
\begin{aligned}
\nabla^2_{\x\x}h(\x,\y)&\succeq\left({\alpha}\|\x\|_{2}^{{d_1}-2}+{\sigma}\right)\left(\frac{\alpha}{d_2}\|\y\|_{2}^{d_2}
+1\right)\eye_{n},
\\
\nabla^2_{\y\y}h(\x,\y)&\succeq\left(\frac{\alpha}{d_1}\|\x\|_{2}^{d_1}
+1\right)\left({\alpha}\|\y\|_{2}^{{d_2}-2}+{\sigma}\right)\eye_{m}    
\end{aligned}
\label{eqn:H:bounds}
\end{equation}
Therefore, we have that  $f$ is $(L_1,L_2)$-relative bi-smooth with respect to $h$ for any
\[
L_1 \ge \frac{1}{L_0}\max\left\{1,\frac{1}{\sigma}, \frac{d_2}{\alpha} \right\}
\text{ and }
L_2 \ge \frac{1}{L_0}\max\left\{ 1,\frac{1}{\sigma}, \frac{d_1}{\alpha} \right\}
\]
with $L_0=\sum_{i=2}^{d_1}\sum_{j=0}^{d_2} i(i-1) \|\calA_{i,j}\|.$

\section{Proof of \Cref{lem:h:exist:2}}
\label{proof:lem:h:exist:2}
\begin{proof}
We will divide the proof into two parts.

\smallskip\paragraph{\bf (1) Showing the relative smoothness condition} By definition, it suffices to show that there is a $L_f>0$ such that
$
L_f\nabla^2h(\x) \pm \nabla^2f(\x)\succeq 0
$
in the whole domain.
In one way, by assumption of $f(\x)$, we have
\[\|\nabla^{2}f(\x)\|\le C_{1}+C_{2}\|\x\|_{2}^{d-2}\]
in the whole domain with $d\ge 2$ for some positive constants $C_{1},C_{2}$.
In another way, by direct computations, $h(\x)$ in \eqref{eqn:h:exist} satisfies that
\[
\nabla^{2}h(\x)\succeq (\alpha \|\x\|_{2}^{d-2}+\sigma)\eye_{n} \text{ for any $d\ge2$},
\]
in the whole domain.
Therefore, it is clear to see that
$
L_f\nabla^2h(\x) \pm \nabla^2f(\x)\succeq 0
$
in the whole domain
for any $L_f\ge\max\{\frac{C_{1}}{\sigma},\frac{C_{2}}{\alpha}\}$.

\smallskip\paragraph{\bf (2) Showing relative bi-smoothness condition} In one way, for any $d_1,d_2\ge 2$,
\begin{align*}
\|\nabla^{2}_{\x\x}f(\x,\y)\|&\le (C_{1}+C_{2}\|\x\|_{2}^{d_1-2})(C_3+C_4\|\y\|_2^{d_2}),
\\
\|\nabla^{2}_{\y\y}f(\x,\y)\|&\le (C_5+C_6\|\x\|_2^{d_1})(C_{7}+C_{8}\|\y\|_{2}^{d_2-2})    
\end{align*}
Then this along this \eqref{eqn:H:bounds} implies that
$f$ is $(L_1,L_2)$-relative bi-smooth w.r.t. $h$
for any 
\[L_1\ge\max\left\{\frac{C_1}{\sigma},\frac{C_2}{\alpha}, C_3, \frac{C_4 d_2}{\alpha} \right\}
\text{ and }
L_2\ge\max\left\{ C_5, \frac{C_6 d_1}{\alpha},\frac{C_7}{\sigma}, \frac{C_8}{\sigma} \right\}.
\]
\end{proof}

\section*{Acknowledgments}The authors gratefully acknowledge Waheed Bajwa, Haroon Raja, Clement Royer, Yue Xie, Xinshuo Yang, and Stephen J. Wright for helpful discussions.

\bibliographystyle{siamplain}
\bibliography{nonconvex}

\end{document}